\documentclass{amsart}

\usepackage[english]{babel}

\usepackage{amsmath,mathrsfs, mathtools, amssymb}
\usepackage{xcolor}
\usepackage{dsfont}

\usepackage[paper=a4paper, margin=3cm]{geometry}

\usepackage[breaklinks]{hyperref}       
\usepackage{url}            
\usepackage{booktabs}       
\usepackage{amsfonts}       
\usepackage{nicefrac}       
\usepackage{microtype}      
\usepackage{lipsum}

\usepackage{amsmath,amssymb,amsthm,amscd}
\usepackage{amsrefs}
\usepackage{tikz,tikz-cd}
\usetikzlibrary{arrows,arrows.meta,calc,patterns}

\usepackage{graphicx}
\graphicspath{}
\DeclareGraphicsExtensions{.pdf,.png,.jpg,.jpeg}

\usepackage{yhmath}
\usepackage{enumitem}
\usepackage{breakurl}

\DeclareMathAlphabet{\mathpzc}{OT1}{pzc}{m}{it}

\newcounter{stmcounter}[section]

\newcounter{thmMaincounter}

\newtheorem{formula}{}[section]
\newtheorem{proposition}[formula]{Proposition}
\newtheorem{corollary}[formula]{Corollary}
\newtheorem{lemma}[formula]{Lemma}
\newtheorem{theorem}[formula]{Theorem}
\newtheorem{theoremM}[thmMaincounter]{Theorem}
\newtheorem{corollaryM}[thmMaincounter]{Corollary}

\theoremstyle{definition}

\newtheorem{definition}[formula]{Definition}
\newtheorem{example}[formula]{Example}

\theoremstyle{remark}

\newtheorem{remark}[formula]{Remark}

\newcommand{\Z}{\mathbb Z}
\newcommand{\R}{\mathbb R}
\newcommand{\C}{\mathbb C}

\newcommand{\N}{\mathbb N}

\newcommand{\T}{\mathbb T}
\newcommand{\Hc}{\mathbb H}
\newcommand{\eps}{\varepsilon}
\newcommand{\lb}{\lbrace}
\newcommand{\rb}{\rbrace}

\newcommand{\Cc}{\mathbf{C}}

\DeclareMathOperator{\Id}{Id}
\DeclareMathOperator{\col}{colim}

\DeclareMathOperator{\Hm}{Hom}

\DeclareMathOperator{\Top}{\mathbf{Top}}

\DeclareMathOperator{\sset}{\mathbf{sSet}}

\DeclareMathOperator{\Set}{\mathbf{Set}}

\DeclareMathOperator{\Oper}{\mathbf{Oper}}

\DeclareMathOperator{\Map}{Map}

\DeclareMathOperator{\Alg}{\mathbf{Alg}}

\DeclareMathOperator{\Cact}{\mathcal{C}act}
\DeclareMathOperator{\Sing}{Sing}
\DeclareMathOperator{\Diag}{Diag}
\DeclareMathOperator{\End}{End}
\DeclareMathOperator{\Flat}{Flat}

\DeclareMathOperator{\dset}{\mathbf{dSet}}
\DeclareMathOperator{\Ass}{Ass}
\DeclareMathOperator{\Conf}{Conf}
\DeclareMathOperator{\dTop}{\mathbf{dTop}}
\DeclareMathOperator{\Nat}{Nat}
\DeclareMathOperator{\Path}{Path}
\DeclareMathOperator{\len}{length}

\renewcommand{\leq}{\leqslant}

\renewcommand{\ge}{\geqslant}

\renewcommand{\phi}{\varphi}

\newcommand{\rcolon}{\mathrel{}\colon\!}

\tikzcdset{arrow style=tikz, diagrams={>=stealth}}

\catcode`,\active

\catcode`\,12

\begin{document}
	
	\author{Eduardo Hoefel}
	\address[E.\,Hoefel]{Universidade Federal do Paran\'a, Brazil}
	\email{hoefel@ufpr.br}
	
	\author{Grigory Solomadin}
	\address[G.\,Solomadin]{Philipps-Universit\"at Marburg, Germany}
	\email{solomadi@mathematik.uni-marburg.de}
	\title{Cellular approximations for nonsymmetric homotopy operads}


	\begin{abstract}
		In this note, for every topological colored nonsymmetric operad with CW spaces in each arity we construct a homotopy isomorphism with a CW homotopy operad (i.e. all higher homotopies are cellular with respect to the product structure and CW structure on cubes) not changing its objects.
		By using the two-sided bar construction, we rectify this to a homotopy isomorphism with a strict CW operad.
		This applies more generally to topological algebras over colored operads in sets.
		In the case of diagrams of finite CW complexes over finite posets, the resulting diagram consists of finite CW spaces enhancing the Quillen adjunction between simplicial sets and topological spaces.
		As a byproduct, we produce an embedding of topological homotopy algebras over operads into the presheaf category over the branch resolution for the category of $P$-trees generalizing the known case of topological operads.
	\end{abstract}
	
	
	\maketitle
	
	\section{Introduction}\label{sec:intro}
	
	The compatibility between a topological operad structure and cellular structures on its components has attracted considerable attention since the problem involving the Fox--Neuwirth cells and the real Fulton--MacPherson operad was noticed. In fact, the cellular decomposition suggested by Fox--Neuwirth cells does not directly give a cellular operad structure compatible with operadic composition, as was clarified after Tamarkin's counterexample and discussed, for instance, in the work of Voronov and Markl \cite{voro2000HGA,markl2015extension}. This question was later revisited from different points of view. In particular, Salvatore constructed a regular cellular decomposition of the two-dimensional Fulton--MacPherson operad compatible with operadic composition \cite{sa-22} (see also \cite{sa-19} for its isomorphism with the $W$-construction of itself). More recently, Rossi and Salvatore studied a related chain-level problem of Koszul duality for the Gravity and Hypercommutative operads, using regular CW-decompositions of the spaces $\overline{\mathcal M}_{0,n+1}$ and pointing again to the subtle question of finding cellular structures compatible with operadic composition \cite{rossi-sa24}.
	
	The present paper is motivated by the following question. Let $O$ be a topological operad whose components are CW-complexes, but whose structure maps are not necessarily cellular with respect to the given CW-structures. Would it be possible to coherently deform the operad structure into a homotopy operad structure which respects the cellular structures? In other words, can one replace $O$, up to a suitable homotopy equivalence, by a CW homotopy operad whose higher structure maps are cellular, while preserving the original cellular structure? The main result of this paper gives a positive answer to this question in the nonsymmetric colored setting. More generally, we prove such a CW-approximation result for topological homotopy algebras over colored operads in sets.
	
	Given a topological nonsymmetric $C$-colored operad $O\in \Oper_{C}\Top$ (where $\Top$ denotes the category of compactly generated weakly Hausdorff topological spaces, and continuous maps between them) for a set $C$, there exists a CW-approximation of $O$, i.e. a CW operad weakly equivalent to the given one.
	Namely, in the Quillen equivalence
	$	\begin{tikzcd}
		{|-|\colon} \sset \arrow[shift left=.5ex]{r} & \Top \colon \Sing \arrow[shift left=.5ex]{l},
	\end{tikzcd} $
	given by the geometric realization and the singular simplicial set functor, both functors are strongly symmetric monoidal, and the counit of the adjunction is a weak equivalence, inducing:
	$
	|\Sing O|\overset{\sim}{\to} O.
	$
	This CW-approximation changes the operad in every arity. 
	The cofibrant replacement using \textit{cellular $P$-algebras} leads to a CW-approximation of  algebras over topological operads~\cite{be-mo-06}.
	
	By definition, for a $C$-colored operad $P$ in sets a \textit{homotopy $P$-algebra} $A$ is a $WP$-algebra with values in a given monoidal category~\cite{we-07}, where $W$ denotes the $W$-construction of Boardman and Vogt.
	By definition, a \textit{homotopy morphism} between homotopy algebras is a homotopy algebra over the Boardman-Vogt product $P\otimes_{BV} \Delta^{1}$ suitably restricting to the leaf and root components, where $\Delta^{1}$ denotes the unary operad consisting of a unique arity $1$ operation~\cite{we-07}.
	If the maps of the homotopy algebra action are homotopy equivalences for all $1_{c}\otimes_{BV}\Delta^{1}$, $c\in C$, then the homotopy morphism is called a \textit{homotopy isomorphism}.
	We call a homotopy algebra $A$ a \textit{CW homotopy algebra} if all higher homotopies of $A$ are cellular maps with respect to the product CW structure of the standard CW structure on the cube $[0,1]^{n}$ and fixed CW structures on $A(c)$, $c\in C$.
	A \textit{CW homotopy operad} is a particular case of a CW homotopy algebra by letting $P=\mathcal{O}_{C}$ be the \textit{operad of $C$-colored operads}~\cite[\S 1.5.6, \S 1.5.7]{be-mo-06}.
	The first main result of this paper is given as follows (see Theorem~\ref{thm:gen}).
	
	\begin{theoremM}\label{thmm:cwapp}
		Let $P$ be any unital, non-symmetric, $C$-colored operad in sets and let $A\in \Alg(WP,\Top)$ be any topological homotopy $P$-algebra such that $A(c)$ is a CW complex for every $c\in C$.
		Then there exists a homotopy isomorphism $A\Rightarrow A_{cw}$ with a CW homotopy $P$-algebra $A_{cw}$ such that $A_{cw}(c)=A(c)$ holds for all $c\in C$.
	\end{theoremM}
	
	The particular cases for Theorem~\ref{thmm:cwapp} are:
	\begin{enumerate}[label=(\roman*)]
		\item any topological strict $P$-algebra $A\in\Alg(P,\Top)$ (as a homotopy algebra with trivial higher homotopies);
		\item any topological $C$-colored operad $O$ (by letting $P=\mathcal{O}_{C}$) (see \S~\ref{sec:app});
		\item any (strict or homotopy coherent) diagram $D$ over a small category $\Cc$ (by considering it as an algebra over the Berger-Moerdijk unary operad $\Diag_{\Cc}$~\cite{be-mo-06}).
	\end{enumerate}
	In particular, we provide a CW homotopy structure on the Fulton-MacPherson operad, addressing Salvatore's question \cite[p.449]{sa-22} 
	in a weaker form.
	Theorem~\ref{thmm:cwapp} also implies the existence of a CW homotopy operad structure on the Deligne-Mumford operad $\{ \overline{M}_{0,n+1} \}_{n \geqslant 2}$, thus answering the weaker form of a question raised by Salvatore and Rossi \cite[\S1.1 (1)]{rossi-sa24}.
	This CW approximation applies to many other well-known topological operads, see~\S\ref{sec:app}.
	The proof of Theorem~\ref{thmm:cwapp} amounts to constructing a certain homotopy isomorphism. 
	This is carried out first by reducing the homotopy $P$-algebra case to that of a homotopy $\Omega(T)$-algebra (Lemma~\ref{lm:red}).
	Second, we fill horns of the cylinder decomposition for $W(\Omega(T)\otimes_{BV}\Delta^{1})$~\cite{ci-mo-11}, which is possible by the existence of all horn fillers (under suitable conditions) in a strict sense for homotopy algebras by following the ideas from~\cite{co-po-86}.
	Third, we use a suitable pasting of classical CW approximation homotopies at the last step of the filler construction.
	
	\begin{corollaryM}\label{corm:rect}
		Under the conditions of Theorem~\ref{thmm:cwapp}, there exists a homotopy isomorphism $A\Rightarrow \widetilde{A}_{cw}$ with a strict CW $P$-algebra $\widetilde{A}_{cw}\in \Alg(P,\Top)$.
	\end{corollaryM}
	
	Corollary~\ref{corm:rect} (see~\S\ref{ssec:pf1}) is obtained by the rectification of homotopy algebras that is given by the geometric realization of monadic two-sided bar construction~\cite{be-mo-06'}.
	By restricting to unary operads with finitely many different compositions, we obtain a finite CW approximation for homotopy coherent diagrams of finite CW complexes.
	In general, even if $A(c)$ is a finite CW complex, $\widetilde{A}_{cw}(c)$ is usually an infinite CW complex (due to the bar construction parameters).
	Notice that a homotopy isomorphism $A\Rightarrow \widetilde{A}_{cw}$ is obtained by filling the inner $2$-horn consisting of $A\Rightarrow A_{cw}$ (Theorem~\ref{thmm:cwapp}) and of the rectification $A_{cw}\Rightarrow \widetilde{A}_{cw}$ of the $WP$-algebra $A_{cw}$, see Lemma~\ref{lm:rectzag}. 
	
	Construction of homotopy isomorphisms by strict horn filling is well known for presheaf categories~\cite{co-po-86}.
	Motivated by this, we prove another main result of this paper (see Theorem~\ref{thm:halgdiag}).
	Here, the category $\Omega P$ consists of $P$-decorated (at vertices) planar rooted trees and their morphisms being compositions of face and degeneracy maps respecting the operad $P$ structure (see~\S\ref{ssec:ptrees}), $\underline{\Omega} P$ is a certain resolution of $\Omega P$ along all branches,
	and the strict Segal condition for algebras is defined in terms of leaves, see~\S\ref{ssec:dneralg}.

	\begin{theoremM}\label{thmm:preshemb}
		Let $P$ be any unital $C$-colored operad in sets.
		Then there is the isomorphism of categories
		\[
			N_{P}^{h}\colon	\Alg(WP,\Top)\to \Top^{(\underline{\Omega} P)^{op}}_{strict}.
		\]
	\end{theoremM}
	
	We call $N_{P}^{h}$ a \textit{homotopy nerve} for topological homotopy algebras over operads in sets (compare with \cite[\S13]{he-mo-22}).
	In the particular case of $P=\mathcal{O}_{C}$, Theorem~\ref{thmm:preshemb} was proved (for symmetric operads) in~\cite[Theorem~B.6]{clrrw-22}.
	In fact, $N_{\mathcal{O}}^{h}$ and the embedding from~\cite[Theorem B.6]{clrrw-22} are related by means of a commutative diagram (see Theorem~\ref{thm:comparison}) involving a flattening functor (see definition~\ref{def:hoflat}).
	
	While Heuts and Moerdijk develop a homotopy theory of algebras by modeling them as left fibrations over the dendroidal nerve~\cite{he-mo-22}, our approach to topological homotopy algebras over operads stays within the realm of strict topological algebras over resolved operads in sets. 
	To our knowledge, the corresponding strict $1$-categorical translation (Theorem~\ref{thmm:preshemb}) has not been explicitly spelled out in the literature.
	For the comparison see Proposition~\ref{pr:algneriso}. 
	
	The presheaf embedding of homotopy (iso)morphisms is obtained by substituting $P\otimes_{BV}\Delta^{1}$ into $P$ in Theorem~\ref{thmm:preshemb}.
	One may hope to obtain Theorem~\ref{thmm:cwapp} by applying classical horn filling~\cite{co-po-86} to strict Segal presheaves over $P\otimes_{BV}\Delta^{1}$.
	However, this idea runs into the problem that there is no cylinder decomposition for the category $\Omega (P\otimes_{BV} \Delta^{1})$ of decorated trees analogous to that of operads.
	
	The structure of the paper is given as follows.
	\begin{itemize}
	\item In~\S\ref{ssec:dset} the formalism of dendroidal categories and fillers is recalled.
    The dendroidal nerve for algebras over operads is defined in terms of the category $\Omega P$ of $P$-trees that is introduced in~\S\ref{ssec:ptrees}.
    In~\S\ref{sec:thm1} the necessary details are given on homotopy morphisms of homotopy algebras in terms of $W$-construction~\S\ref{ssec:homoalg}, their rectification using two-sided bar construction~\S\ref{ssec:rectif}, and dendroidal formalism~\S\ref{ssec:dset}.
	The definition of CW-algebras (in strict and homotopy sense) is given in~\S\ref{ssec:cw}.
	The remaining contents of~\S\ref{sec:thm1} concern the existence of horn fillers~\S\ref{ssec:hornext}, dendroidal cylinder decomposition~\S\ref{ssec:cyldec} resulting in the proof of Theorem~\ref{thmm:cwapp} (see \S\ref{ssec:pf1}).
	\item After a lemma on morphisms of branches (\S\ref{ssec:branch}), the category $\underline{\Omega} P$ is defined in terms of metric functions and their $*$-product in~\S\ref{ssec:resop}.
	Associativity of the $*$-product is proved in Lemma~\ref{lm:assoc}.
	A reduction of metric functions to those corresponding to leaf embeddings is given in Lemma~\ref{lm:leafcomp}.
	In~\S\ref{ssec:dneralg}, the definition of a dendroidal nerve for strict algebras, Segal condition and proof of nerve isomorphism (Proposition~\ref{pr:dalg}) are given. 
	The comparison with the similar (but different) notion of nerve from~\cite{he-mo-22} is in Proposition~\ref{pr:algneriso}.
	In~\S\ref{def:nhp}, the homotopy nerve of homotopy algebras is defined.
	After a couple of technical lemmas, Theorem~\ref{thmm:preshemb} (i.e. the homotopy nerve isomorphism, see Theorem~\ref{thm:halgdiag}) is proved.
	~\S\ref{ssec:comphorns} contains a comparison of horns for homotopy algebras and their homotopy nerves (i.e. of presheaves).
	
	\item We recall the embedding from~\cite{clrrw-22} for strict and homotopy topological operads (defined in terms of the explosion category) in~\S\ref{ssec:strictop} and~\S\ref{ssec:homop}, respectively.
	In~\S\ref{ssec:facaux} we discuss pruned vertices in Lemma~\ref{lem:bookkeeping}, and introduce $\mathcal{O}_{C}$-trees associated with a decomposition of a dendroidal morphism. 
	The poset of paths from~\cite{clrrw-22} is compared with another poset in Lemma~\ref{lm:pi}, see~\S\ref{ssec:comppos}.
	In~\S\ref{ssec:compemb}, the homotopy flattening functor $\Flat^{h}$ is defined, and two embeddings for homotopy operads are described in terms of it.
	(The explicit description of $\Flat^h$ uses the auxiliary trees mentioned above.)
	The description in the case of strict algebras is given in Remark~\ref{rem:strict}.

	\item In the final Section~\ref{sec:app} we briefly outline the applications of our method, and propose further directions of study.
	\end{itemize}

	\noindent {\bf Acknowledgements:} 
	The authors thank N.~Wahl and M.~Robertson for providing additional details on the proof of~\cite[Theorem B.6]{clrrw-22}. 
	We gratefully acknowledge funding of the Deutsche Forschungsgemeinschaft
	(DFG, German Research Foundation): Project number 561158824 (Walter Benjamin Fellowship of G.S.).
    G.S. is grateful to K.~Kuribayashi for the helpful discussions of this project and hospitality during G.S.'s research visit to Shinshu University.
    Finally, we thank M.~Frankland and P.~Salvatore for the attention to this work.
	
	\section{Proof of Theorem~\ref{thmm:cwapp}}\label{sec:thm1}

	\subsection{Dendroidal sets and $\infty$-operads}\label{ssec:dset}

    Let $C$ be a set of colors.	
	In the sequel we only consider operads $P$ that are \textit{unital} (i.e. $P(c;c)$ contains a unit of all infinitesimal compositions), which is required for the definition of degeneracies in various constructions below.
	In this subsection we follow the book of Heuts and Moerdijk \cite{he-mo-22}.
	Consider the dendroidal category $\Omega_{C}$ of planar rooted $C$-colored trees as objects.
	A tree is a pair $(V_T, H_T)$ of its vertices $V_T$ and half-edges $H_T$ (open edges called \emph{leaves} $L(T)$ of $T$, and orbits of orientation involution that are called \emph{inner edges} $ie(T)$ of $T$).
	By definition, an \textit{outer (or external) vertex} $v$ of a tree $T$ has exactly one inner incident edge; otherwise, $v$ is called an \textit{inner vertex} of the tree $T$.
	
	Furthermore, any top (bottom, respectively) outer vertex $v\in V_{T}$ is called a \textit{leaf vertex} (\textit{root vertex}, respectively).
	A tree without vertices is denoted by $\eta=\eta_{c}$, where $c\in C$ is the color of its unique edge.
	Any tree from $\Omega_{C}$ is equipped with a coloring of its edges by elements of $C$, i.e. with a map $c_{T}\colon H_T\to C$ compatible with involutions of inner edges in $T$.
	The morphisms in this category are compositions of the following maps:
	\begin{enumerate}[label=(\roman*)]
		\item an \emph{inner face} map $\partial_{e}\colon \partial_{e}T\to T$ for any inner edge $e$ of $T$ (edge $e\in ie(T)$ expansion);
		\item an \emph{outer face} map $\partial_{v}\colon \partial_{v}T\to T$ for any external vertex $v$ of a tree $T$ with at least two vertices, where $\partial_{v}T$ is obtained from $T$ by pruning away $v$ and all outer edges attached to it (embedding as a subgraph).
		If $T=C_{v}$ is a corolla then the respective outer faces $\eta\to C_{v}$ are given by the embedding of a leaf or the root open edge;
		\item a \emph{degeneracy} $\sigma_{e}\colon \sigma_{e}T \to T$ for any edge $e$ of $T$ (removing a new vertex of valence $2$ corresponding to a unary operation of the same color), where $\sigma_{e}T$ is obtained from $T$ by adding a new vertex in the middle of the edge $e$.
	\end{enumerate}
	
	For any such tree $T$ there is the free $C$-colored operad $\Omega(T)\in \Oper_{C}\Set$ generated by vertices of $T$, and the compositions are given by grafting.
	A \emph{dendroidal object} in a bicomplete monoidal category $\Cc$ is by definition any object from $\Cc^{\Omega^{op}_{C}}$.
	(For example, dendroidal sets form the presheaf category $\dset:=\Set^{\Omega_{C}^{op}}$~\cite{we-07,he-mo-22}.)
	Denote by $\Omega[T]\in \dset$ the representable presheaf in sets, i.e.
	\[
	\Omega[T]:=\Hm_{\Omega_{C}}(-,T).
	\]
	There is the realization-nerve adjoint pair of functors~\cite[\S 3.5.1]{he-mo-22}.
	\begin{equation}\label{eq:nrpair}
		\tau\colon \dset\leftrightarrows \Oper_{C}\Set\rcolon N,
	\end{equation}
	\[
	NP(T):=\Hm(\Omega(T),P),\ \tau(X):=\underset{\Omega[T]\to X}{\col} \Omega(T).
	\]
	Notice that 
	\[
	\tau \Omega[T]=\Omega(T),\ 
	N \Omega(T)=\Omega[T],
	\]
	hold for any tree $T$.
	The functor $N$ is called the \textit{dendroidal nerve} of an operad.
	Let
	\[
	\partial \Omega[T]:=\underset{S\to T}{\col}\ \Omega[S],
	\]
	where the colimit is taken with respect to all faces $\partial_{x} T$ of $T$, if $T$ has at least two vertices.
	In particular, $x$ is either any internal edge $e$ or any external vertex $v$ of $T$. The case of a corolla $T=C_{v}$ corresponds to edges, and the case of an edge $T=\eta$ corresponds to the empty dendroidal set. 
	
	We briefly recall the definition of classical infinity-operads and their horn fillers that is necessary for~\S\ref{ssec:cyldec}.
	For any tree $T$, the presheaf
	\begin{equation}\label{eq:colhornd}
		\Lambda_{e}[T]:=\col_{e\neq e'} \partial_{e'} T\in \dset,
	\end{equation}
	is called the \textit{inner horn} for an internal edge $e$ of $T$~\cite[\S 6]{he-mo-22}.
	If the root vertex $r$ of $T$ has exactly one inner edge attached to it, then omitting the root vertex $r$ one obtains a \textit{root horn} $\Lambda_{r} [T]$ of $T$.
	\textit{Leaf horn} of $T$ is defined for every leaf vertex $v\in V_T$ in a similar way.
	Root and leaf horns are called \textit{external horns}.
	A horn of a dendroidal set $X$ is a morphism $\Lambda_{x}[T]\to X$.
	For a horn $\Lambda_{x}[T]$ and a dendroidal set $X$, its \textit{filler} is a map fitting into the commutative diagram similar to that from definition~\ref{defn:alghorn}.
	A dendroidal set admitting all internal horn fillers is called an \textit{$\infty$-operad}~\cite{we-07}.

	\subsection{Category $\Omega P$}\label{ssec:ptrees}
	Let $P$ be a unital $C$-colored operad in sets.
	Let $\Omega P$ be the category of elements for $P$; specifically, its
	objects are operad morphisms $\psi_{T}\colon \Omega(T)\to P$ for all planar rooted trees $T$.
	Equivalently, an object of $\Omega P$ is any planar rooted tree with $C$-colored edges and $P$-decorated vertices, which we call a \emph{$P$-tree} and denote by $\mathcal{T}=(T,\psi_{T})$.
	A morphism $\mathcal{S}=(S,\psi_{S})\to\mathcal{T}$ in $\Omega P$ is a morphism $f\colon S\to T$ in $\Omega_{C}$ such that the following diagram commutes
	\begin{equation}
		\begin{tikzcd}
			\Omega(S) \arrow{r}{f_{*}}\arrow[swap]{rd}{\psi_{S}} & \Omega(T)\arrow{d}{\psi_{T}}\\
			& P.
		\end{tikzcd}
		\label{cd_mophism_omegaP}
	\end{equation}	
		
	There is the forgetful functor $U\colon \Omega P\to \Omega_{C}$ given by forgetting all vertex decorations of a $P$-tree.  
	Any morphism in $\Omega P$ is uniquely determined by the target $P$-tree $\mathcal{T}$ and an arrow in $\Omega_{C}$.
	In particular, any morphism in $\Omega P$ is a composition of those corresponding to 
	inner, outer faces and degeneracy maps.
	These generators correspond to the operadic composition, projection forgetting the arguments, and substitution of a unit in respective vertex labels.
	Figure \ref{fig:inner_face_map} illustrates a morphism in $\Omega P$ in the case where $f$ is an inner face map.

	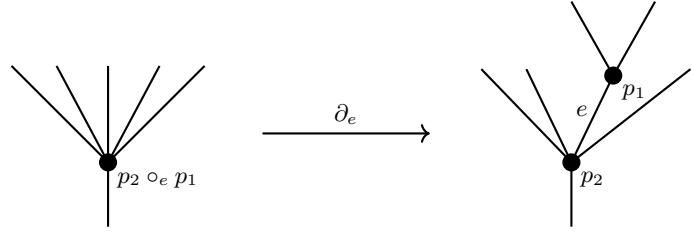
\begin{figure}
	\begin{center}
	\begin{tikzpicture}[scale=0.85, every node/.style={font=\small}]
				
				
				\coordinate (vL) at (0,0);
				\coordinate (rL) at (0,-1.0);
				
				\coordinate (l1L) at (-1.5,1.5);
				\coordinate (l2L) at (-0.8,1.5);
				\coordinate (l3L) at (0,1.5);
				\coordinate (l4L) at (0.8,1.5);
				\coordinate (l5L) at (1.5,1.5);
				
				\draw[thick] (vL) -- (rL);
				\draw[thick] (vL) -- (l1L);
				\draw[thick] (vL) -- (l2L);
				\draw[thick] (vL) -- (l3L);
				\draw[thick] (vL) -- (l4L);
				\draw[thick] (vL) -- (l5L);
				
				\fill (vL) circle (4pt);
				
				\node[below right] at (vL) {$p_2 \circ_e p_1$};

				
				\draw[->, thick] (2.4,0.45) -- (5.0,0.45)
				node[midway, above] {$\partial_e$};

				
				\coordinate (T) at (7.2,0);
				
				\coordinate (v2) at ($(0,0) + (T)$);
				\coordinate (rR) at ($(0,-1.0) + (T)$);
				
				\coordinate (v1) at ($(0.65,1.35) + (T)$);
				
				\coordinate (a1) at ($(-1.4,1.45) + (T)$);
				\coordinate (a2) at ($(-0.7,1.45) + (T)$);
				\coordinate (a3) at ($(1.85,1.45) + (T)$);
				
				\coordinate (b1) at ($(0.0,2.5) + (T)$);
				\coordinate (b3) at ($(1.3,2.5) + (T)$);
				
				\draw[thick] (v2) -- (rR);
				\draw[thick] (v2) -- (a1);
				\draw[thick] (v2) -- (a2);
				\draw[thick] (v2) -- (a3);
				\draw[thick] (v2) -- (v1);
				
				\draw[thick] (v1) -- (b1);
				\draw[thick] (v1) -- (b3);
				
				\fill (v2) circle (4pt);
				\fill (v1) circle (4pt);
				
				\node[below right] at (v1) {$p_1$};
				\node[below right] at (v2) {$p_2$};
				
				\node[left] at ($(0.4,0.80) + (T)$) {$e$};
				
			\end{tikzpicture}
			\caption{The inner face map $\partial_e$ defines a morphism in $\Omega P$ if the vertex decorations are compatible with the operad structure of $P.$}
			\label{fig:inner_face_map}
		\end{center}
	\end{figure}

	\subsection{Homotopy morphisms of homotopy algebras over operads}\label{ssec:homoalg}
	
	Recall that the $W$-construction $WP$ equips the inner edges of trees with lengths (whose tuple corresponds to a point in a cube),
	and the equivalence relations for zero length edges.
	The metric tree $\T$ consists of a $P$-decoration at its vertices and of lengths of its inner edges (in $[0,1]$).
	In particular, after forgetting the $P$-decoration of $\T$, the resulting tree is homeomorphic to the cube $I^{m}$, where $m$ is the number of inner edges in $\T$.
	The structure maps of $WP$ satisfy the following conditions~\cite{be-mo-06'}: 1) a metric tree $\T$ with length $0$ edge $e$ is identified with $\T/e$ and composed boundary vertices using the operad composition; 2) for any unary operation vertex, one removes the vertex and takes the maximum of the lengths of the adjacent edges.
	
	\begin{definition}
		An \emph{algebra} $A$ in $\Top$ over a $C$-colored operad $P$ in sets is a $C$-colored object $\lb A(c)\rb_{c\in C}\in \Top^{C}$ in topological spaces with structure (continuous) maps
		\[
		P(c_{1},\dots,c_{n};c)\times A(c_{1})\times\cdots\times A(c_{n})\to A(c),\ 
		c_{i},c\in C,
		\]
		satisfying infinitesimal associativity conditions
		\[
		(P\circ P)\circ A=P\circ (P\circ A)
		\]
		with respect to the compositional product $\circ$.
	\end{definition}
	
	Given any $C$-colored object $X\in \Cc^{C}$ in a complete closed monoidal category $\Cc$, there is the endomorphism operad $\End X\in \Oper_{C} \Cc$ given by
	\[
	(\End X)(c_{1},\dots,c_{n};c):=
	\Map_\Cc(\prod_{i} X(c_{i}),X(c)),
	\]
	whose infinitesimal compositions are defined by functional substitutions (i.e. evaluation and substitution in $\Cc$), where $\Map_\Cc(-,-)$ is the internal hom for $\Cc$.
	For any operad $P$ in $\Cc$ denote by $\Alg_{X}(P,\Cc)$ the set of all $P$-algebras on the $C$-colored object $X$.
	By tensor-hom adjunction one has
	\begin{equation}\label{eq:algoper}
		\Alg_{X}(P,\Cc)\cong \Oper_{\Cc}(P,\End X).
	\end{equation}
	The category of operads (in a monoidal category) is monoidal with respect to the Boardman-Vogt tensor product $\otimes_{BV}$~\cite{be-mo-06'}.
	As it is well known, a morphism $f\colon A\to B$ between $P$-algebras can be considered as an algebra over $P\otimes_{BV}\Delta^{1}$.
	Here $\Delta^{1}$ has a unique (unary) uncolored operation $0\to 1$, $0\to_{c} 1:=1_{c}\otimes_{BV} (0\to 1)$ and $1_{c}$ is the operadic $c$-colored unit for $P$.
	Denote the operads with grafted $\Delta^{1}$-operations at leaves and root by $P\otimes_{BV} 0$ and $P\otimes_{BV} 1$, respectively.
	Then $f|_{P\otimes_{BV} 0}=A$, $f|_{P\otimes_{BV} 1}=B$.
	
	\begin{definition}[{\cite[Ch.~4]{we-07}, \cite[\S2.9]{ma-sh-st-00}}]
		For any $C$-colored operad $P$ in sets, an algebra over $WP$ in $\Top$ is called a \textit{homotopy algebra} over $P$.
		For any homotopy algebras $A,B$ over $P$ a \textit{homotopy morphism} $f\colon A\Rightarrow B$ is a homotopy algebra $f\in \Alg(W(P\otimes_{BV}\Delta^{1}),\Top)$ such that 
		$f|_{W(P\otimes_{BV} 0)}=A$, $f|_{W(P\otimes_{BV} 1)}=B$ hold.
		A homotopy morphism is called a \textit{homotopy isomorphism}, if $f(0\to_{c} 1)$ is a homotopy equivalence for every color $c\in C$.
	\end{definition}
	
	The category $\Alg(WP,\Top)$ of strict algebras over the operad $WP$ is canonically enriched over $\Top$.
	The structure maps of a $WP$-algebra $A$ are given by 
	\begin{equation}\label{eq:algstr}
		\phi_{\mathcal{T}}\colon I^{m}\times \prod_{l\in L(T)}A(l)\to A(r),
	\end{equation}
	where $\mathcal{T}=(T,\psi)$ is any $P$-decorated (at its vertices) planar tree with $m$ inner edges, $L(T)$ denotes the set of leaves in $T$, and $r=r(T)$ is the root of $T$.
	By taking the adjoint map to~\eqref{eq:algstr}, the higher homotopy of $A$ corresponding to the $P$-tree $\mathcal{T}$ is equivalently given by the map
	\begin{equation}\label{eq:thetamap}
		\theta_{\mathcal{T}}\colon I^{m}\to \End UA,
	\end{equation}
	where $UA$ denotes the $C$-colored object of the algebra $A$.
	
	We explain the labeling of the $2m$ facets for the cube $I^{m}$.
	Every inner edge $e\in ie(T)$ contributes exactly by two facet labels.
	The first label is $\theta_{\mathcal{T}/e}$.
	Let $\mathcal{B}$ ($\mathcal{U}$, respectively) be the tree obtained from $T$ by pruning all vertices $v$ above $i(e)$ (not above $i(e)$, respectively), where $i(e)$ is the vertex of the edge $e$ that is closer to leaves of the tree $T$, and the tree is oriented from root to leaves.
	The second label is $\theta_{\mathcal{B}}\circ \theta_{\mathcal{U}}$.
	(For more details on $P$-trees see~\S\ref{ssec:ptrees}.)
	
	\subsection{Rectification of algebras over operads}\label{ssec:rectif}
	
	There is the free-forgetful adjunction
	\begin{equation}\label{eq:ffadjalg}
		F\colon \Cc^{C}\leftrightarrows \Alg(P,\Cc) \rcolon U.
	\end{equation}
	The category $\Alg(P,\Top)$ of topological algebras has the closed model category transferred from the projective model category on $C$-colored objects via \eqref{eq:ffadjalg}.
	If such a transfer is well defined then the category $\Alg(P,\Top)$ is called \emph{admissible}.
	Admissibility for algebras in a sufficiently nice monoidal category (e.g. in $\Top$) over nonsymmetric operads in sets was proved by Muro~\cite{mu-11}.
	Since the Quillen closed model category on $\Top$ is left proper, by~\cite[Theorem~1.5]{mu-11} there is a Quillen equivalence
	\[
	\eps_{!}\colon \Alg(W P,\Top)\leftrightarrows
	\Alg(P,\Top)\rcolon \eps^{*},
	\]
	where the functor $\eps^{*}$ is the precomposition with the augmentation map $\eps\colon WP\to P$.
	Let 
	\[
	\eps_{!}(A):=|B_{\bullet}(P,WP,A)|
	\]
	be the topological $P$-algebra with the induced left $P$-action on the two-sided bar construction.
	(We refer to~\cite{ri-14},~\cite{be-mo-06'} for the standard properties of the two-sided, or double-sided, bar construction.)
	The functor $\eps_{!}$ is called \textit{a rectification} of a homotopy $P$-algebra~\cite{be-mo-06}.
	The functor $\eps_{!}$ represents the derived left adjoint of $\eps^{*}$.
	In particular, there is a homotopy isomorphism $A\Rightarrow \eps_{!}A$ (the unit adjunction) where by a slight abuse of the notation we omit $\eps^{*}$ for any $W_{\bullet}P$-algebra $A$.
	This corresponds to a natural zigzag of weak equivalences given below.
	
	Notice that the geometric realization preserves weak equivalences if the bar construction is a Reedy cofibrant simplicial object (e.g. see~\cite[Ch.14]{ri-14}). 
	Furthermore, it preserves homotopy equivalences for CW spaces objectwise by the classical Whitehead theorem (the homotopy inverse maps do not form a strict algebra morphism, in general, but they do form a homotopy isomorphism~\cite{be-mo-06}).
	
	\begin{lemma}\label{lm:rectzag}
		There is a natural zigzag of weak equivalences for $W P$-algebras
		\[
		\begin{tikzcd}
			\mid B_{\bullet}(P,WP,A)\mid &
			\mid B_{\bullet}(WP,WP,A)\mid \arrow[swap]{l}{\mid\eps_{*}\mid}\arrow{r}{\sim} &
			A.
		\end{tikzcd}
		\]
	\end{lemma}
	\begin{proof}
		Let $M := WP \circ -$ be the monad on $\Top^C$ associated with the topological operad $WP$. 
		A homotopy $P$-algebra is equivalently an algebra over the monad $M$. The map $|B_\bullet(WP, WP, A)| \to A$ is the geometric realization of the standard monadic bar resolution. 
		Because $A$ is a $WP$-algebra, the augmented simplicial space $B_\bullet(WP, WP, A) \to A$ possesses an extra degeneracy (a simplicial contracting homotopy) induced by the monadic unit $A \to WP(A)$. 
		Therefore, its geometric realization is homotopy equivalent (and thus weakly equivalent) to $A$.
		
		The augmentation of the Boardman-Vogt resolution yields a map of topological operads $\varepsilon \colon WP \xrightarrow{\sim} P$, which is a weak equivalence. 
		This induces a map of simplicial spaces:
		\[
		\varepsilon_* \colon B_\bullet(WP, WP, A) \to B_\bullet(P, WP, A).
		\]
		In simplicial degree $n$, this map evaluates the respective operads on a specific collection:
		\[
		\varepsilon_n \colon WP \circ (WP)^{\circ n} \circ A \to 	P \circ (WP)^{\circ n} \circ A.
		\]
		For any collection of CW spaces $X$, the evaluation functor $- \circ X$ preserves weak equivalences of operads, as it is constructed via coproducts of spaces multiplied by Cartesian powers of $X$. 
		Consequently, $\varepsilon_n$ is a weak equivalence of spaces for every $n$, meaning $\varepsilon_*$ is a levelwise weak equivalence.
		
		Finally, because both $B_\bullet(WP, WP, A)$ and $B_\bullet(P, WP, A)$ are Reedy cofibrant simplicial spaces (a standard property of the monadic bar construction for CW operads and algebras), and geometric realization preserves levelwise weak equivalences between Reedy cofibrant simplicial spaces, the induced map $|\eps_*|$ is a weak equivalence. 
		This completes the proof.
	\end{proof}
	
	\subsection{CW algebras}\label{ssec:cw}
	
	We work in the category of CGWH spaces where the Cartesian product of CW complexes is a CW complex (with compactly generated topology).
	
	\begin{definition}\label{def:cwhalg}
		We say that a topological $P$-algebra $A$ is a \textit{CW $P$-algebra} if $A(c)$ is a CW-complex, $c\in C$, and every structure map $\phi_{\mathcal{T}}$ (see~\eqref{eq:algstr}) for $A$ is cellular with respect to the CW-structure on the respective Cartesian product.
		We call a homotopy algebra $A\in\Alg(WP,\Top)$ a \textit{CW homotopy algebra} if $A$ is a CW-algebra over $WP$.
	\end{definition}
	
	\begin{remark}
		Notice that for any operad $P$ in sets, the Boardman-Vogt resolution $WP$ is a CW operad in topological spaces (defined as a CW algebra over operad of operads, see~\S\ref{sec:compoper}), and a CW algebra over the operad of operads, because its spaces of operations $WP(c_{1},\dots,c_{n};c)$ are disjoint unions of cubes (over all trees with fixed colors of the leaves and the root), and the composition maps are defined by facet inclusions, which are cellular maps in the standard CW structure of the cube.
	\end{remark}
	
	\begin{corollary}\label{cor:algrect}
		For any CW homotopy $P$-algebra $A$, $\eps_{!}(A)$ is a CW $P$-algebra.
	\end{corollary}
	\begin{proof}
		Write the geometric realization of a simplicial space $X_{\bullet}:=B_{\bullet}(P,WP,A)$ as a coequalizer
		\[ 
		|X_{\bullet}|=
		\col \left( \coprod_{f \colon [m] \to [n]} X_{n} \times 	\Delta^m \rightrightarrows \coprod_{k \ge 0} X_{k} \times \Delta^k \right).
		\]
		Then the claim follows, since both arrows are cellular, because all face (i.e. algebra structure maps) and degeneracy (i.e. identity maps, corresponding to units of the operad $P$ which is unital by the assumption) maps are cellular by the assumption.
	\end{proof}
	
	\begin{remark}
		Notice that one cannot formulate definition~\ref{def:cwhalg} using the adjoint structure maps $\theta_{\mathcal{T}}$ (see \eqref{eq:thetamap}), since the mapping space between CW spaces is not a CW space, in general.
	\end{remark}
	
	\subsection{Cylinder decomposition for dendroidal sets}\label{ssec:cyldec}
	
	We follow~\cite{ci-mo-11}.
	Consider the $C$-colored cylinder
	\[
	\Omega[T]\otimes\Delta^{1}:=
	N(\Omega(T)\otimes_{BV}\Delta^{1})
	\]
	for any $C$-colored tree $T$, where by abuse of the notation we denote the representable presheaf $N\Delta^{1}$ for the linear tree with one vertex by $\Delta^{1}$.
	It is convenient to additionally mark the vertices of $T$ by white color and the vertex of $\Delta^{1}$ (unary operation) by black color.
	There is the collection of all trees obtained by inserting unary operations from $\Delta^{1}$ to $T$ such that any path from any leaf to the root has exactly one such operation (the resulting trees are called \textit{percolation diagrams} in~\cite{ci-mo-11}).
	This collection has the partial order, where the least element $T_{1}$ has black unary operations grafted at leaf vertices, and the greatest element $T_{N}=\sigma_{er} T$ has the black root vertex $r$ (of valence $2$), where $er$ is the root edge of $T$.
	Consider any linear extension of this partial order, so that the above collection is a sequence $T_{1},\dots,T_{N}$.
	Let 
	\[
	A_{0}:=
	\partial\Omega[T]\otimes\Delta^{1}\cup \Omega[T]\otimes 0,
	\]
	\[
	A_{i}:=A_{0} \cup\ \bigcup_{j\leq i} \Omega[T_{i}],\ 
	i=1,\dots,N.
	\]
	One has 
	\[
	\Omega[T]\otimes\Delta^{1}=
	A_{N}=
	A_{0} \cup \bigcup\limits_{j=1}^{N} \Omega[T_{i}].
	\]
	
	\begin{lemma}[{\cite[Lemma~5.6]{ci-mo-11}}]\label{lm:inter}
		For any $i\neq j$ one has 
		\[
		\Omega[T_{i}]\cap \Omega[T_{j}]\subseteq \bigcup\limits_{k\leq i,j} \Omega[T_{k}].
		\]
	\end{lemma}
	
	Following standard notation from model category theory, an inclusion of dendroidal objects is called \textit{inner anodyne} if it is obtained as an iterated pushout of inner horn fillings.
	
	\begin{lemma}[{\cite[Theorem~5.2]{ci-mo-11}}]\label{lm:fill}
		\begin{enumerate}[label=(\roman*)]
			\item The inclusion $A_{i-1}\to A_{i}$ is inner anodyne for any $0<i<N$;
			\item There is the equality $A_{N}=\Omega[T_{N}]\cup A_{N-1}$, in which $A_{N}$ is obtained from $\Lambda_{r}[T_{N}]\subseteq A_{N-1}$ by the outer (root) horn filling.
		\end{enumerate}
	\end{lemma}
	
	\subsection{Horn fillers for homotopy algebras over operads}\label{ssec:hornext}
	
	For brevity, we define horns for algebras here only in the case of free operads.
	Horns for algebras over arbitrary operads are defined in a similar way using $P$-trees, which requires a more heavy notation and will be used in Proposition~\ref{pr:horncorr} only.
    (We use the tautological $\Omega(T)$-labeling $\psi_{can}$ of any vertex $v\in iv(T)$ of the tree $T$ by the respective corolla $C_v$, and abuse slightly the notation by writing $\theta_{T}$ instead of $\theta_{\mathcal{T}}$, $\mathcal{T}=(T,\psi_{can})$ for a planar tree $T$.)

	\begin{definition}\label{defn:alghorn}
		For a homotopy algebra $A\in \Alg(W\Omega(T),\Top)$ and a tree $T$ the map 
		\begin{equation}\label{eq:hornalg}
			\Lambda_{x}(A,T)\colon \Lambda_{x}(T)\to \End UA, 
		\end{equation}
		is called the \emph{horn} of $(A,T)$, where $\Lambda_{x}(T)$ is the \emph{$x$-box}, i.e. the union of all facets of the cube $I^{m}$ except the one corresponding to $x$.
		Here, if $x=e$ is an inner edge of $T$ then the horn is called \emph{inner} and the facet label is $\theta_{T/e}$.
		If $x=v$ is an external vertex of $T$ then the horn is called \emph{outer} and the facet label is $\theta_{B}\circ \theta_{U}$.
		(Formally, 
		\begin{equation}\label{eq:colhornop}
			\Lambda_{x}(T)=\underset{S\in \partial_{x} T}{\col}\  W\Omega(S)
		\end{equation}
		denotes the colimit in the category of operads of the resolved free operads $W\Omega(S)$ over all dendroidal faces $S$ of $T$ except the face $x$, compare with~\S\ref{ssec:dset}.)
		For a horn $\Lambda_{x}(A,T)$ of a homotopy algebra $A$, its \textit{filler} is a map fitting into the commutative diagram
		\[
		\begin{tikzcd}
			\Lambda_{x}(T)\arrow{r}{\Lambda_{x}(A,T)}\arrow[hook]{d} & \End UA.\\
			W\Omega(T)\arrow[dashed]{ru}
		\end{tikzcd}
		\]
	\end{definition}
	
	\begin{remark}
		The horn $\Lambda_{x}(A,T)$ of a homotopy $P$-algebra $A$ is in natural correspondence with the horn of the respective presheaf $N_{P}^{h} A$ (see Theorem~\ref{thm:halgdiag}, Proposition~\ref{pr:horncorr}). This explains our choice of terminology.
	\end{remark}
	
	We prove the following lemma based on the construction from~\cite{co-po-86},~\cite[Theorem 4.2]{ci-mo-11}.
	
	\begin{lemma}\label{lm:exthorn}
		Let $A$ be a homotopy algebra over $\Omega(T)\in\Oper_{C}\Set$ for $T\in \Omega_{C}$.
		\begin{enumerate}[label=(\roman*)]
			\item For any inner edge $e\in ie(T)$ the inner horn $\Lambda_{e}(A,T)$ of $A$ has a filler.
			\item For any external vertex $v\in V_{T}$ suppose that it is unary and the induced map $\theta_{C_{v}}$ (see~\eqref{eq:thetamap}) is a homotopy equivalence.
			In addition, if $v$ is the root vertex, assume that there is a unique inner edge incident to $v$ in $T$.
			Then the external horn $\Lambda_{v}(A,T)$ of $A$ has a filler.
		\end{enumerate}
	\end{lemma}
	\begin{proof}
		In $(i)$ one chooses any (e.g. linear deformation retraction from the cube to the respective box) filler of the box corresponding to the facet $\theta_{T/e}$ of the cube $\theta_{T}$ establishing the existence of the inner horn filler.
		
		In $(ii)$ we consider only the case when $v=r$ is the root vertex of $T$ (the leaf case is proved in a similar way).
		Let $U=T\setminus r$ with the root vertex $r'$.
		Since $r$ is a unary vertex by the assumption, $l(T)=l(U)$.
		Introduce
		\[
		X:=\Map(\prod_{l\in L(T)} A(l),A(r')),\ 
		Y:=\Map(\prod_{l\in L(T)} A(l),A(r)).
		\]
		Here by a slight abuse of the notation we denote the root $r$ edge of $T$ by $r$.
		Let $m$ be the number of inner edges in the tree $T$.
		To establish a filler for the horn $\Lambda_{r}(A,T)\colon \Lambda_{r}(T)\to Y$ one must find a map $\theta_{T}\colon I^{m}\to Y$ extending $\Lambda_{r}(A,T)$, and another map $\theta_{U}\colon I^{m-1}\to X$ such that 
		\begin{equation}\label{eq:homcomdec}
		\theta_{T}|_{I^{m-1}}=\theta_{C_{r}}\circ \theta_{U}
		\end{equation}
		holds.
		One can explicitly construct the strict topological filler by applying the BV cylinder filler in unary case, exactly as it is done by Cordier and Porter~\cite[Theorem~3.1]{co-po-86},
		since $\theta_{C_{r}}$ is a homotopy equivalence.
		This finishes the proof.
	\end{proof}
	
	\begin{remark}
		Note that the assumption in Lemma~\ref{lm:exthorn} is necessary: if the root had multiple incident inner edges, pruning it would yield a forest, meaning the respective dendroidal horn is undefined in $\Omega_{C}$.
	\end{remark}
	
	\subsection{Proof of Theorem~\ref{thmm:cwapp}}\label{ssec:pf1}
	
	We may reduce any homotopy $P$-algebra to the case of a homotopy algebra over a free operad by the following lemma.
	
	\begin{lemma}\label{lm:red}
		For any operad $P$ in sets one has
		\[
		\Alg(WP,\Top)\cong \lim_{\Omega[T]\to NP} \Alg(W\Omega(T),\Top),
		\]
	\end{lemma}
	\begin{proof}
		It is enough to establish the bijection for a fixed $C$-colored object $X$ in $\Top$ by taking union over all possible ones.
		Writing the presheaf $NP$ as a colimit of representables one has
		\[
		W_{!}(NP)\cong
		\underset{\Omega[T]\to NP}{\col} W\Omega(T)\cong
		\underset{\Omega(T)\to P}{\col} W\Omega(T)\cong	
		WP,
		\]
		where $W_{!}\colon \dset \to \Oper\Top$ is the left Kan extension of $T\mapsto W\Omega(T)$.
		Therefore, 
		\[
		\Oper(WP,\End X)\cong
		\Oper\left(\underset{\Omega[T]\to NP}{\col} 	W\Omega(T),\End X\right)\cong
		\underset{\Omega[T]\to NP}{\lim} \Oper(W\Omega(T),\End 	X).
		\]
		It remains to apply~\eqref{eq:algoper} to both sides of the last equivalence.
	\end{proof}
	
	Now we can prove a slightly generalized version of Theorem~\ref{thmm:cwapp}.
	
	\begin{theorem}\label{thm:gen}
		\begin{enumerate}[label=(\roman*)]
			\item For any homotopy algebra $A\in\Alg(WP,\Top)$ and any collection $u_{c}\colon A(c)\to B(c)$ of homotopy equivalences, where $B(c)$ is a CW complex for any $c\in C$, there exist a CW homotopy algebra $B\in\Alg(WP,\Top)$ with objects $\lb B(c)\rb_{c\in C}$ and a homotopy isomorphism $f\colon A\Rightarrow B$ satisfying $f(1_{c})=u_{c}$, $c\in C$.
			\item For any homotopy algebra $B\in\Alg(WP,\Top)$ and any collection $u_{c}\colon A(c)\to B(c)$ of homotopy equivalences, where $A(c)$ is a CW complex for any $c\in C$, there exist a CW homotopy algebra $A\in\Alg(WP,\Top)$ with objects $\lb A(c)\rb$ and a homotopy isomorphism $f\colon A\Rightarrow B$ satisfying $f(1_{c})=u_{c}$, $c\in C$.
		\end{enumerate}
	\end{theorem}
	\begin{proof}
		By Lemma~\ref{lm:red}, it is enough to consider the case $P= \Omega(T)$ for any tree $T$.
		If $T$ has a unique inner vertex then the claim holds by the homotopy equivalence assumption.
		Using the cylinder decomposition (lemmas~\ref{lm:inter},~\ref{lm:fill}) and Lemma~\ref{lm:exthorn}~$(i)$, one finds inner horn fillers (defining the higher homotopies for the algebra structure) for all trees in the cylinder decomposition except the last one, see \eqref{eq:homcomdec}.
		At the last step one finds the external horn filler by Lemma~\ref{lm:exthorn}~$(ii)$, because by the condition $u_{c}$ are homotopy equivalences for all $c\in C$.
		
		We modify the obtained structure map of the homotopy algebra $A$ as follows.
		Take the adjoint 
		\[
		\phi_{T}|_{I^{m-1}}=\phi_{C_{r}}\circ \phi_{U}
		\]
		to~\eqref{eq:homcomdec} using exponentiation.
		Since $\phi_{U}$ is a map between CW complexes, there exists a cellular approximation relative to the boundary~\cite[Theorem~4.8]{ha-01}.
		Consider the cellular approximation for the restriction of $\phi_{U}$ to $I^{m-1}\times A_{in}\to A_{out}$ relative to $\partial I^{m-1}\times A_{in}$, where
		\[
		A_{in}:=\prod_{l\in l(T)}A(l),\ 
		A_{out}:=A(r(U)).
		\]
		Paste the homotopy into $I^{m}\times A_{in}$ to obtain the desired map (changing it neither on $\partial I^{m-1}\times A_{in}$ nor on the $r$-box), see fig.\ref{fig:hompas}.
		Then the restriction is cellular on $I^{m-1}\times A_{in}$, which completes the construction.
		
		The proof of $(ii)$ is similar to the above, see~\cite[Appendix B]{ci-mo-11} for the relevant details.
	\end{proof}
	
	\begin{figure}
		\centering
		\begin{minipage}[t]{0.1\textwidth}
			\centering
			\begin{tikzpicture}
				\draw[thick,blue,pattern=horizontal lines,pattern color=blue] (1,0) -- (0.5,0.5)--(1,1);
				\draw[thick, draw=black] (0,0) -- (1,0) -- (0.5,0.5) -- (1,1) -- (0,1) -- (0,0);
				\draw[thick, draw=black] (0,0) rectangle ++(1,1);
			\end{tikzpicture}
		\end{minipage}
		\begin{minipage}[t]{0.1\textwidth}
			\centering
			\begin{tikzpicture}
				\draw[thick, draw=black] (0,0) -- (1,0) -- (0.5,0.5) -- (1,1) -- (0,1) -- (0,0);
			\end{tikzpicture}
		\end{minipage}
		\begin{minipage}[t]{0.1\textwidth}
			\centering
			\begin{tikzpicture}
				\draw[thick,fill=blue, opacity=0.2] (1,0) -- (0.5,0.5)--(1,1);
				\draw[thick, draw=green] (1,0) -- (1,1);
				\draw[thick, draw=black] (0,0) -- (1,0) -- (0.5,0.5) -- (1,1) -- (0,1) -- (0,0);
			\end{tikzpicture}
		\end{minipage}
		\caption{Pasting of a relative CW-approximation homotopy to the higher homotopy $\phi_{U}$.
		}\label{fig:hompas}
	\end{figure}
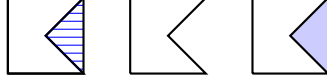
	
	Corollary~\ref{corm:rect} about rectification of CW homotopy algebras (see~\S\ref{sec:intro}) is a direct consequence of the following theorem.
	
	\begin{theorem}\label{thm:cwapp}
		Let $P$ be any unital $C$-colored operad in sets.
		Then for every topological algebra $A\in \Alg(P,\Top)$ that is a CW complex in each color, there exists a homotopy isomorphism $A\Rightarrow \widetilde{A}_{cw}$ with a strict topological CW algebra $\widetilde{A}_{cw}$.
	\end{theorem}
	\begin{proof}
		Theorem~\ref{thm:gen} produces a homotopy isomorphism $A\Rightarrow A_{cw}$ with a CW homotopy algebra $A_{cw}$.
		Fill the $2$-horn consisting of the zigzag between $A_{cw}$ and its rectification $\widetilde{A}_{cw}:=\eps_{!} A_{cw}$ (from Lemma~\ref{lm:rectzag}) to get a homotopy isomorphism $A\Rightarrow \widetilde{A}_{cw}$.
		By corollary~\ref{cor:algrect}, $\widetilde{A}_{cw}$ is a CW algebra.
		This completes the proof.
	\end{proof}
	
	\begin{remark}
		In the case of a unary finite operad (i.e. corresponding to a category with finitely many simplices in the nerve), the respective dendroidal and simplicial nerves coincide and have only finitely many dimensions where non-degenerate simplices occur.
		Then the respective bar construction is homotopy equivalent to its finite skeleton.
		Hence, under these assumptions in Theorem~\ref{thm:cwapp} one obtains finite-dimensional CW complexes $\widetilde{A}_{cw}(c)$, provided that $A(c)$ is a finite CW complex for every $c\in C$.
	\end{remark}
	
	\section{Proof of Theorem~\ref{thmm:preshemb}}\label{sec:thm3}

	\subsection{Branches}\label{ssec:branch}

	We write $E(T)$, $ie(T)$, $L(T)$, $r(T)$ for the edges, inner edges, leaves and root of a
	tree $T$, and $T[e]:=T_{\geqslant e}$ be the \textit{branch} of $T$ at $e\in E(T)$, respectively. 

	\begin{lemma}\label{lm:branches}
		Let $f\colon S\to T$ be a morphism in $\Omega_{C}$.
		\begin{enumerate}[label=(\roman*)]
			\item One has $L(f(S))=f(L(S))$, and $f$ is injective on $L(S)$. 
			In particular, the branches $T[f(l)]$,
			$l\in L(S)$, are pairwise disjoint.
			\item One has $T[f(r(S))]=f(S)\circ\bigl(T[f(l)]\bigr)_{l\in L(S)}$, the grafting being along the
			leaves of $f(S)$.
			\item For any composition $S\xrightarrow{f}U\xrightarrow{g}T$ and any leaf $l\in L(S)$,
			\[
			T[gf(l)]=g\bigl(U[f(l)]\bigr)\circ\bigl(T[g(m)]\bigr)_{m\in L(U[f(l)])}.
			\]
		\end{enumerate}
	\end{lemma}
	\begin{proof}
		$(i)$ A degeneracy identifies only the two edges adjacent to the unary vertex it
		deletes, and these are comparable, so at most one of them is a leaf; faces are
		injective. Hence $f$ is injective on $L(S)$, and, $f$ being order preserving,
		$f(L(S))$ consists of maximal elements of $f(S)$, i.e. $f(L(S))\subseteq L(f(S))$.
		
		Conversely, let $a\in L(f(S))$. If $a$ is created by $f$, say
		$a\in ie\bigl(f(C_v)\bigr)$ for some $v\in iv(S)$, then an
		input edge of $f(C_v)$ lies strictly above $a$, contradicting maximality.
		Hence, $a=f(x)$ for some $x\in E(S)$. Assume $x\notin L(S)$ and let $v$ be the vertex of $S$
		above $x$. 
		If $f(C_v)\neq\eta$, it has an edge strictly above $f(x)=a$, again
		contradicting maximality. 
		Otherwise, $f$ deletes the unary vertex $v$ and identifies $x$ with its input edge $y$, so that $a=f(y)$ and $x$ may be replaced by $y$. 
		As $S$ is finite, iterating gives $a\in f(L(S))$.
		Two distinct leaves of $f(S)$ are incomparable in $f(S)$, hence in $T$, and no edge lies above two incomparable
		edges; this gives the disjointness.
		
		$(ii)$ The inclusion $\supseteq$ holds, because $f$ preserves the order. 
		Let $e\in E(T[f(r(S))])$ and choose
		a path $f(r(S))=e_0<\dots<e_k=e$. Let $i$ be maximal with $e_i\in E(f(S))$, which exists since
		$f(r(S))\in E(f(S))$. 
		If $i=k$, then $e\in E(f(S))$. 
		Otherwise, $e_{i+1}\notin E(f(S))$.
		We claim that $e_i$ is a leaf of $f(S)$: indeed, otherwise the vertex of $T$ above $e_i$ lies in $f(S)$ together with all its input edges, which contradicts maximality.
		Hence, $e_{i+1}\in E(f(S))$. 
		By $(i)$, one has $e_i=f(l)$ for some $l\in L(S)$, and
		$e\geqslant e_{i+1}>f(l)$.
		
		$(iii)$ The branch $U[f(l)]$ is up-closed, so $g$ restricts on it to a morphism $U[f(l)]\to T$ with root
		$g(f(l))$.
		Now apply $(ii)$ to this restriction.
		This finishes the proof.
	\end{proof}
		
	\subsection{Category $\underline{\Omega} P$}\label{ssec:resop}
	
	We introduce the indexing category in order to define presheaves that recover
	homotopy algebras. 
	The higher homotopies $\varphi_{\mathcal{T}}$ of a $WP$-algebra are parameterized by the lengths of the inner edges of $\mathcal{T}$, the facets $t_e=0$ and $t_e=1$
	corresponding to the contraction of $e$ and to a decomposition at $e$ respectively.
	A leaf carries the length $1$, the value it acquires as soon as a branch is grafted onto it;
	accordingly we let a metric be defined on all non-root edges and be equal to $1$ on the leaves.
	(By definition, leaves have no root, and the root does not receive length.)
	We use the convention
	$\max\varnothing:=0$.
	For $f\colon S\to T$ in $\Omega_{C}$ let
	\[
		E_f(T):=\{a\in E(T)\colon \ a>f(l)\ \text{for some}\ l\in L(S)\}
	\]
	be the set of edges lying strictly above the image of a leaf; by lemma \ref{lm:branches} (i) the leaf
	$l$ is then unique. Let $M_f$ be the space of maps $\mu\colon E_f(T)\to I=[0,1]$ with $\mu\equiv 1$ on
	$L(T)\cap E_f(T)$. Restricting to inner edges identifies
	\[
		M_f\cong\ 
		\prod_{l\in L(S)} I^{ie(T[f(l)])}.
	\]
	Finally, for $\mu\in M_f$ and $a\in E(T')$, where $g\colon T\to T'$, put
	\[
		(g_*\mu)(a):=\max\bigl\lb\mu(e')\colon 	e'\in E_f(T),\ g(e')=a\bigr\rb.
	\]
	
	\begin{definition}\label{def:thick}
		The category $\underline{\Omega} P$ has the same objects as $\Omega P$, and
		\[
			 \underline{\Omega}P(\mathcal{S},\mathcal{T}) := \coprod_{f\in\Omega P(\mathcal{S},\mathcal{T})} M_{f}.
		\]
		The composition of 
		$
		\begin{tikzcd}
		\mathcal{S} \arrow{r}{(f,\mu)} &
		\mathcal{U} \arrow{r}{(g,\nu)} &
		\mathcal{T}
		\end{tikzcd}
		$ is defined by
		\[
			(g,\nu)\circ(f,\mu):=(gf,\nu*\mu),
		\]
		where for $a\in E_{gf}(T)$, lying above $gf(l)$ with
		$l\in L(S)$, 
		\begin{equation}\label{eq:comp}
			(\nu*\mu)(a):=
			\begin{cases}
				(g_*\mu)(a), & a\in E(g\bigl(U[f(l)]\bigr)),\\
				\nu(a), & \text{otherwise.}
			\end{cases}
		\end{equation}
	\end{definition}
	
	In the second case of~\eqref{eq:comp}, $a$ lies strictly above a leaf $g(m)$ of the respective tree with $m\in L(U)$, so that $\nu(a)$ is defined. 
	Further, $\nu*\mu\in M_{gf}$: a leaf of $T$ in the second case has $\nu(a)=1$, and a leaf of $T$ in the first case is a leaf of $g\bigl(U[f(l)])$, hence of the form $g(m)$ with $m\in L(U[f(l)])$ by Lemma~\ref{lm:branches} $(i)$, so that $\mu(m)=1$ occurs in $(g_*\mu)(a)$. 
	Since $\max$ is continuous, the
	category $\underline{\Omega} P$ is topologically enriched.
	
	\begin{remark}\label{rem:types}
		We explain \eqref{eq:comp}, which is forced by the defining relations of the $W$-construction, in
		terms of the edges of $g\bigl(U[f(l)])$ (see figure~\ref{fig:multtree}).
		\begin{enumerate}[label=(\roman*)]
			\item an edge $g(e)$, $e\in ie(U[f(l)])$, has as its preimages the inner edges mapping
			onto it, and the deletion of a degenerate vertex on an edge amounts to taking the maximum of the two
			lengths
			(if $g$ is injective, then the maximum is taken over a single edge, and is the identity map);
			\item an edge of the expansion $ie(g(C_w))$, $w\in iv(U[f(l)])$, is created by
			$g$ and has no preimage, so it receives the length $\max\varnothing=0$, i.e. the contraction at that
			edge, see figure~\ref{fig:inner_face_map};
			\item a grafting edge $g(m)$, $m\in L(U[f(l)])$, has the leaf $m$ among its preimages, whence
			the length $\mu(m)=1$; 
			\item the second case of~\eqref{eq:comp} leaves the lengths unchanged along the inner edges of
			the grafted branches $T[g(m)]$.
		\end{enumerate}
		Types $(i)$ and $(iii)$ meet only when $g$ is a degeneracy identifying an inner edge with the image of a
		leaf; type $(ii)$ is disjoint from both.
	\end{remark}

	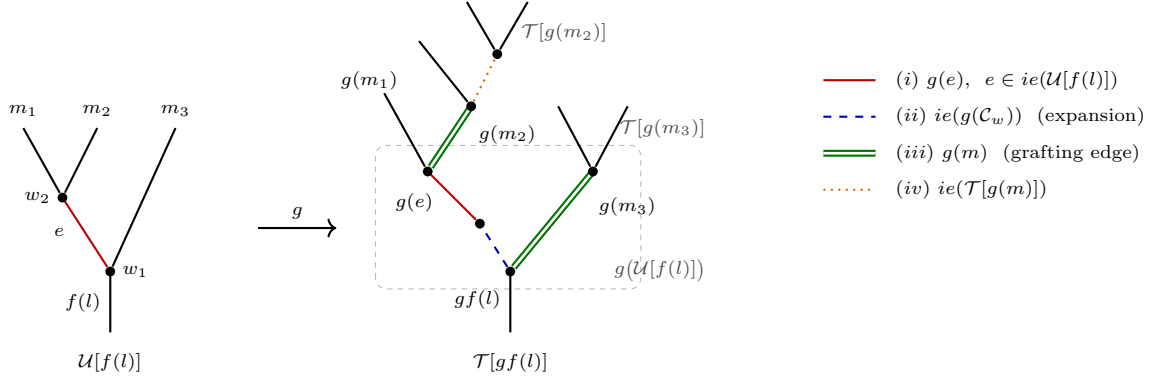
\begin{figure}
	\centering
	\begin{tikzpicture}[
		scale=1.15,
		edge/.style   ={thick},
		k1/.style     ={thick, red!75!black},
		k2/.style     ={thick, blue!70!black, dashed},
		k3/.style     ={thick, green!45!black, double, double distance=1pt},
		k4/.style     ={thick, orange!85!black, dotted},
		vert/.style   ={circle, fill=black, inner sep=1.3pt},
		lbl/.style    ={font=\scriptsize},
		reg/.style    ={rounded corners, draw=gray!55, dash pattern=on 2pt off 2pt}
		]
		
		\begin{scope}[shift={(-4.6,0)}]
			\node[lbl] at (0,-1.35) {$\mathcal U[f(l)]$};
			
			\draw[edge] (0,-1) -- (0,-0.3);
			\node[lbl, left] at (0,-0.65) {$f(l)$};
			
			\node[vert] (w1) at (0,-0.3) {};
			\node[lbl, right=1pt] at (w1) {$w_1$};
			
			\draw[k1] (w1) -- (-0.55,0.55);
			\node[lbl, left] at (-0.42,0.15) {$e$};
			
			\node[vert] (w2) at (-0.55,0.55) {};
			\node[lbl, left=1pt] at (w2) {$w_2$};
			
			\draw[edge] (w2) -- (-1.0,1.35);
			\draw[edge] (w2) -- (-0.15,1.35);
			\draw[edge] (w1) -- (0.75,1.35);
			\node[lbl] at (-1.0,1.55) {$m_1$};
			\node[lbl] at (-0.15,1.55) {$m_2$};
			\node[lbl] at (0.75,1.55) {$m_3$};
		\end{scope}
		
		\draw[->, thick] (-2.9,0.2) -- (-2.0,0.2) node[midway, above, lbl] {$g$};
		
		\begin{scope}
			\node[lbl] at (0,-1.35) {$\mathcal{T}[ gf(l)]$};
			
			\draw[reg] (-1.55,-0.5) rectangle (1.5,1.15);
			\node[lbl, gray!60!black] at (1.72,-0.28)
			{$g\bigl(\mathcal{U}[ f(l)]\bigr)$};
			
			\draw[edge] (0,-1) -- (0,-0.3);
			\node[lbl, left] at (0,-0.62) {$gf(l)$};
			
			\node[vert] (a1) at (0,-0.3) {};
			\draw[k2] (a1) -- (-0.35,0.25);           
			\node[vert] (a2) at (-0.35,0.25) {};
			
			\draw[k1] (a2) -- (-0.95,0.85);
			\node[lbl, left] at (-0.78,0.5) {$g(e)$};
			\node[vert] (b1) at (-0.95,0.85) {};
			
			\draw[edge] (b1) -- (-1.45,1.75);
			\node[lbl] at (-1.62,1.9) {$g(m_1)$};
			
			\draw[k3] (b1) -- (-0.45,1.6);
			\node[lbl, right=1pt] at (-0.5,1.28) {$g(m_2)$};
			\node[vert] (c1) at (-0.45,1.6) {};
			\draw[k4] (c1) -- (-0.15,2.2);            
			\node[vert] (c2) at (-0.15,2.2) {};
			\draw[edge] (c2) -- (-0.5,2.8);
			\draw[edge] (c2) -- (0.2,2.8);
			\draw[edge] (c1) -- (-1.05,2.35);
			\node[lbl, gray!60!black] at (0.62,2.45) {$\mathcal{T}[ g(m_2)]$};
			
			\draw[k3] (a1) -- (0.95,0.85);
			\node[lbl, right=1pt] at (0.88,0.45) {$g(m_3)$};
			\node[vert] (d1) at (0.95,0.85) {};
			\draw[edge] (d1) -- (0.55,1.6);
			\draw[edge] (d1) -- (1.35,1.6);
			\node[lbl, gray!60!black] at (1.75,1.35) {$\mathcal{T}[ g(m_3)]$};
		\end{scope}
		
		\begin{scope}[shift={(3.6,1.9)}]
			\draw[k1] (0,0)    -- (0.6,0);    \node[lbl, right] at (0.72,0)
			{$(i)$ $g(e)$, \ $e\in ie(\mathcal{U}[ f(l)])$};
			\draw[k2] (0,-0.42) -- (0.6,-0.42); \node[lbl, right] at (0.72,-0.42)
			{$(ii)$ $ie(g(\mathcal{C}_w))$ \ (expansion)};
			\draw[k3] (0,-0.84) -- (0.6,-0.84); \node[lbl, right] at (0.72,-0.84)
			{$(iii)$ $g(m)$ \ (grafting edge)};
			\draw[k4] (0,-1.26) -- (0.6,-1.26); \node[lbl, right] at (0.72,-1.26)
			{$(iv)$ $ie(\mathcal T[g(m)])$};
		\end{scope}
	\end{tikzpicture}
	\caption{Four types of edges defined by a morphism in $\Omega P$.
	Here, $g$ is a composition of an embedding (several outer face maps) and an inner face placing a degenerate vertex on one of the leaves for the root vertex. The $P$-decorations of vertices are irrelevant for~\eqref{eq:comp} and are omitted for brevity }\label{fig:multtree}
	\end{figure}

	Let $E^{+}(T):=E(T)\setminus\lb r(T)\rb$.

	\begin{lemma}[Associativity]\label{lm:assoc}
		For any composition $\mathcal S\xrightarrow{(f_1,\mu_1)}\mathcal U_1\xrightarrow{(f_2,\mu_2)}\mathcal U_2
		\xrightarrow{(f_3,\mu_3)}\mathcal U_3$ in $\underline\Omega P$ one has
		\begin{equation}\label{eq:assocuop}
			(\mu_3*\mu_2)*\mu_1=\mu_3*(\mu_2*\mu_1).
		\end{equation}
	\end{lemma}
	\begin{proof}
		The underlying tree morphisms agree, both being $f:=f_3f_2f_1$. 
		Fix $l\in L(S)$ and put
		\[
			A:=U_1[f_1(l)],\qquad V:=f_2(A),\qquad W:=f_3(V)=(f_3f_2)(A),\qquad W':=f_3\bigl(U_2[f_2f_1(l)]\bigr).
		\]
		We compare the values of two sides for~\eqref{eq:assocuop} on any $a\in E^{+}\bigl(U_3[f(l)]\bigr)$. 
		Consider three cases.
		
		\begin{enumerate}[label=(\roman*)]
		\item \textbf{$a\in E(W)$.} 
		Both sides of~\eqref{eq:assocuop} on $a$ are then equal $\bigl((f_3f_2)_*\mu_1\bigr)(a)$, the left one by
		\eqref{eq:comp} for $(f_3f_2,f_1)$. The right one equals
		\[
			\max\lb(\mu_2*\mu_1)(b)\colon b\in E_{f_2f_1}(U_2),\ f_3(b)=a\rb.
		\] 
		If $b\in E(V)$, then
		$(\mu_2*\mu_1)(b)=\bigl((f_2)_*\mu_1\bigr)(b)$, and these terms give $\bigl((f_3f_2)_*\mu_1\bigr)(a)$
		because $(f_3f_2)^{-1}(a)=\bigcup_{b\in f_3^{-1}(a)}f_2^{-1}(b)$. Otherwise $b$ lies strictly above a
		leaf $f_2(l')$ of $V$; since $a=f_3(b)$ lies in $W$ above the leaf $(f_3f_2)(l')$ of $W$, and the
		leaves of $W$ are maximal in $W$, one has $a=(f_3f_2)(l')$. Then $\mu_1(l')=1$ occurs on the left and
		$(\mu_2*\mu_1)(f_2(l'))=1$ on the right, so both sides equal $1$, all lengths being $\leqslant1$.
		
		\item \textbf{$a\in E(W')\setminus E(W)$.} 
		Both sides of~\eqref{eq:assocuop} on $a$ are equal to
		\[
			\max\lb\mu_2(b)\colon b\in E^{+}(U_2[f_2(l')]),\ f_3(b)=a\rb.
		\]
		Let $b\in E_{f_2f_1}(U_2)$ with $f_3(b)=a$. Then $b\notin E(V)$, since otherwise
		$a=f_3(b)\in E(W)$; so by the second case of (2) for $(f_2,f_1)$ the edge $b$ lies
		strictly above a leaf $f_2(l'')$ of $V$, with $l''\in L(A)$, and $(\mu_2*\mu_1)(b)=\mu_2(b)$.
		Both $(f_3f_2)(l')$ and $(f_3f_2)(l'')$ lie below $a$, hence are comparable; being leaves of $W$ they are maximal in $W$, so they coincide, and $l''=l'$ by Lemma~\ref{lm:branches} $(i)$. 
		This gives the right-hand side value. 
		For the left-hand side, $a\in E(W')\setminus E(W)$ forces
		$a\in f_3\bigl(U_2[f_2(l')]\bigr)$, so $(\mu_3*\mu_2)(a)=\bigl((f_3)_*\mu_2\bigr)(a)$, the
		maximum being taken over $\lb b\in E_{f_2}(U_2)\colon f_3(b)=a\rb$. 
		Any such $b$ lies strictly above
		$f_2(k)$ for a unique $k\in L(U_1)$, and $k\in L(A)$: otherwise $f_2(k)$ would lie in the branch $U_2[f_2f_1(l_0)]$ of some leaf $l_0\ne l$ of $S$, and $a=f_3(b)$ would lie in
		$U_3[f(l_0)]$, which is disjoint from $U_3[f(l)]\ni a$ by ~\ref{lm:branches} $(i)$. 
		A similar argument on maximality and comparability from the above then gives $k=l'$, so the two index sets
		\[
			\lb b\in E_{f_2}(U_2)\colon f_3(b)=a\rb,\ 
			\lb b\in E_{f_2 f_{1}}(U_2)\colon f_3(b)=a\rb,
		\]
		coincide.
		
		\item \textbf{$a\notin E(W')$.} Then $a$ lies strictly above a leaf of $W'$, so the right-hand side of~\eqref{eq:assocuop} on $a$ equals
		$\mu_3(a)$. 
		Also $a\notin f_3\bigl(U_2[f_2(l')]\bigr)$, this tree being contained in $W'$, so $a$
		lies strictly above a leaf of it and $(\mu_3*\mu_2)(a)=\mu_3(a)$ by~\eqref{eq:comp} for $(f_3,f_2)$.
		\end{enumerate}
	\end{proof}

	We record several important properties of the category $\underline{\Omega} P$ in the following lemma.
	For $x\in E(T)$ let $\iota_{\mathcal{T},x}\colon\eta_{c(x)}\to\mathcal{T}$ be the inclusion of
	the edge $x$, and put $\iota_{\mathcal{T}}:=\iota_{\mathcal{T},r(T)}$.
	There is the projection functor $\pi\colon \underline{\Omega} P\to\Omega P$, $(f,\mu)\mapsto f$.
	
	\begin{lemma}\label{lm:leafcomp}
		For any $(f,\mu)\in\underline{\Omega} P(\mathcal{S},\mathcal{T})$ and $l\in L(S)$, $(\mathcal{T}[f(l)],\mu_{l})$ is an element of the operad $WP$ (i.e. a  $P$-tree with metric on its inner edges), and
		\[
			(f,\mu)\circ(\iota_{\mathcal{S},l},\varnothing)=(\iota_{\mathcal{T},f(l)},\mu_{l})
		\]
		holds, denoting by
		$\mu_{l}$ the restriction of $\mu$ to $E^{+}(T[f(l)])$. 
	\end{lemma}
	\begin{proof}
		The first claim holds trivially, since $\mu_{l}$ assigns a length to every inner edge of the $P$-tree $\mathcal{T}[f(l)]$. 
		For the second claim, the underlying tree morphisms agree, both being the inclusion of the edge $f(l)$.
		Since $S[l]=\eta_l$, one has $E_{\iota_{\mathcal{S},l}}(S)=\varnothing$ and $f(S[l])=\eta_{f(l)}$, so that in~\eqref{eq:comp} the first case is empty and every edge of
		$E_{f\iota_{\mathcal{S},l}}(T)=E(T[f(l)])\setminus\{f(l)\}$ falls into the second case, receiving the length $\mu(a)$.
	\end{proof}

	\begin{example}\label{ex:two-vertex}
		Consider the $1$-colored operad $\Ass$, i.e. the operad on one binary operation $m$ satisfying the associativity condition. 
		Let $\mathcal T$ be the two-vertex $\Ass$-tree, with $m$ at the root vertex $u$ and $m$
		grafted at its first input (vertex $v$), so that $ie(T)=\lb e\rb$.
		There is the component $I^{\lb e\rb}$ of the disjoint union $\underline{\Omega} \Ass(\eta_{r(T)},\mathcal{T})$, and the two endpoints of the segment are
		exactly the composites of non-identities in $\underline\Omega \Ass$:
		\begin{itemize}
			\item $t_e=0$: consider
			\[
			\begin{tikzcd}
			\eta_{r} \arrow{r}{(\iota_{\mathcal{T}/e},\varnothing)} &
			\mathcal{T}/e
			\arrow{r}{(\partial_e,\varnothing)} &\mathcal{T}.
			\end{tikzcd}
			\] 
			Here $e$ is created by
			$\partial_e$ expanding the unique vertex, hence is of type $(ii)$ and receives
			length $0$.
			\item $t_e=1$: consider
			\[
			\begin{tikzcd} \eta_{r}\arrow{r}{(\iota_{\mathcal{C}_{u}},\varnothing)} & 
			C_{u} \arrow{r}{(\partial_{v},\varnothing)} &
			\mathcal{T}.
			\end{tikzcd}
			\] 
			Here $e$ is a grafting edge, hence of type $(iii)$ and receives length $1$.
		\end{itemize}
	\end{example}
			
	\subsection{Nerve for algebras over operads}\label{ssec:dneralg}
	
	Let $P$ be an operad with unit in sets.
	The nerve of an algebra over operad was given in~\cite[\S9.5]{he-mo-22} as the embedding (taking the enriched version in topological spaces) 
	\[
		N(P,-)\colon \Alg(P,\Top)\to 	\Top^{\Omega^{op}_{C}} / NP.
	\]
	We define the functor $N_{P}$ in terms of
	Boardman's cherry trees notion~\cite{bo-vo-73}.
	The cherry tree in figure~\ref{fig:chtree} illustrates an element of the free $P$-algebra generated by the $C$-colored object $\{ A(c) \}_{c \in C}$, where $a_i \in A(c_i)$.
	We show that $N(P,-)$ and $N_{P}$ are isomorphic in Proposition~\ref{pr:algneriso}.
	In the next definition, as usual in dendroidal sets theory, the colors $c_i$ are taken to be the leaves themselves.
		
	\begin{figure}
		\begin{center}
			\begin{tikzpicture}[scale=0.5, every node/.style={font=\small}]
				
				\node[draw, circle, thick, minimum width=1mm, minimum height=1mm] (a1) at (-3,2.4) {$a_1$};
				\node[draw, circle, thick, minimum width=1mm, minimum height=1mm] (a2) at (-0.8,2.4) {$a_2$};
				\node at (0.7,2.4) {$\cdots$};
				\node[draw, circle, thick, minimum width=1mm, minimum height=1mm] (an) at (2.2,2.4) {$a_n$};
				
				\coordinate (v) at (0,0);
				\fill (v) circle (5pt);
				\node at (0.42,-0.22) {$p$};
				
				\draw[thick] (a1.south) -- (v); 
				\draw[thick] (a2.south)  -- (v);
				\draw[thick] (an.south) -- (v); 
				\draw[thick] (v) -- (0,-1.5);
				
				\node at (-1.5,0.5) {$c_1$};
				\node at (-0.10,1.1) {$c_2$};
				\node at (1.4,0.5) {$c_n$};
				\node at (0.3,-1) {$c$};
				
			\end{tikzpicture}    
		\end{center}
		\caption{A cherry tree (corolla) with $p\in P(c_1, \dots, c_n;c)$ and $a_{i}\in A(c_{i})$}\label{fig:chtree}
	\end{figure}
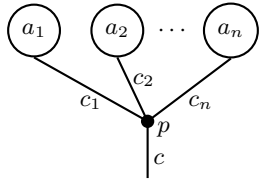
	
	\begin{definition}\label{def:denn}
		Define the \textit{dendroidal nerve for $P$-algebras} functor
		\[
			N_{P}\colon \Alg(P,\Top)\to \Top^{(\Omega P)^{op}},
		\]
		on a $P$-algebra $A$ by mapping any $P$-tree $\mathcal{T}=(T,\psi_{T})$ to the Cartesian product for decorations of the respective $C$-colored leaves by $A(c)$:
		\begin{equation}\label{eq:algner}
			N_{P}A(\mathcal{T}):=
			\prod_{l\in l(T)} A(l).
		\end{equation}
	\end{definition}
	
	The $P$-algebra structure on $A$ is recovered from the values of the presheaf $N_{P}A$ on external face maps.
	E.g. inclusion of the root $\eta_{c}$ to a $P$-decorated corolla $P(c_{1},\dots,c_{n};c)$ corresponds in $\Top^{(\Omega P)^{op}}$ to the map of the labels
	\[
		P(c_{1},\dots,c_{n};c)\times \prod_{i=1}^{n} A(c_{i})\to A(c).
	\]
	Because the value of $N_P A(\mathcal{T})$ depends only on the leaves of the tree $T$, the presheaf $N_{P} A$ maps an inner face map (which expands an inner edge and does not change the respective leaves) to the identity maps. 
	This enforces the associativity of the algebra action with respect to the operadic composition in $P$. 
	The values on degeneracies (which insert unary identity vertices) enforce the unit axioms of $A$.

	\begin{remark}
	For a $P$-tree $\mathcal{T}=(T,\psi_{T})$, let $P(\mathcal{T})$ be the suboperad of $P$ generated by all $P$-corollas in $\mathcal{T}$.
	The $P$-tree $\mathcal{T}$ corresponds to the representable object $\Omega[T] \xrightarrow{\psi_{T}} NP$. 
	Let $F_{P}(\mathcal{T})$ be the free $P(\mathcal{T})$-algebra generated by the leaves of $T$. 
	There is the adjunction~\cite[\S9.5, p.~394]{he-mo-22}
	\[
		F_{P}\colon \dTop/NP \leftrightarrows 	\Alg(P, \mathbf{Top})\rcolon N(P, -).
	\]
	By evaluating this free-forgetful adjunction, one obtains the natural isomorphism:
	\[
		\Hm_{\Alg(P(\mathcal{T}), \Top)}(F_{P}(\mathcal{T}), A) \cong \prod_{l\in L(T)} A(l),
	\]
	i.e., the operations at the internal vertices are completely determined by the $P$-decoration $\psi_{T}$, leaving only the inputs of the algebra $A$ at the leaves as the free variables.
\end{remark}

	\begin{definition}\label{def:segal}
		A presheaf $X\in\Top^{(\Omega P)^{op}}$ satisfies the strict \emph{Segal condition for algebras} if 
		the natural map induced by inclusions of leaf edges $\eta_{l}\to T$ 
		is an isomorphism for every $\mathcal{T}\in \Omega P$:
			\[
				X(\mathcal{T})\overset{\cong}{\to}\prod_{l\in L(T)} X(\eta_{l}).
			\]
		The leaf Segal condition for presheaves over $\underline{\Omega} P$ is defined similarly to the above.
	\end{definition}
	
	\begin{remark}\label{rem:simseg}
	In contrast with the usual dendroidal Segal condition for operads, where a tree is reconstructed from its vertex corollas, the Segal condition for a $P$-algebra $A$ implies that the value on a $P$-tree depends only on the respective leaves. 
	\end{remark}
	
	Denote by $\Top^{(\Omega P)^{op}}_{strict}$ the subcategory of strict Segal presheaves in $\Top^{(\Omega P)^{op}}$.
	
	\begin{proposition}\label{pr:dalg}
		For any operad $P$ in sets the dendroidal nerve functor is an isomorphism of categories
		\[
			N_{P}\colon \Alg(P,\Top)\to \Top^{(\Omega P)^{op}}_{strict}.
		\]
	\end{proposition}
	\begin{proof}
		Bijectivity on objects follows from the reconstruction of the $P$-algebra structure from any strict Segal presheaf, as explained above. 
		Full faithfulness follows from the fact that any natural transformation between strict presheaves is uniquely determined by its components on the edges $\eta_c$. Naturality with respect to the root face maps of corollas is exactly the condition that these components commute with the algebra action.
	\end{proof}

	By the adjunction \eqref{eq:nrpair} between the free operad functor and the dendroidal nerve, a morphism of dendroidal sets $\Omega[T] \to NP$ is equivalent to an operad morphism $\Omega(T)\to P$, relating to the construction of dendroidal nerve for algebras from \cite[\S9.5]{he-mo-22}.
	Hence,
	\[
		\Omega P=\int_{\Omega_{C}} NP.
	\]
	The equivalence of the two definitions for the nerve of an algebra (namely, $N(P,-)$ from~\cite[\S9.5]{he-mo-22} and $N_{P}$ from definition~\ref{def:denn}) follows directly by the following lemma.
	It may be considered as an enriched isomorphism version of the classical equivalence for set-valued presheaves on the category of elements, e.g. see~\cite[exercise~III.8(a)]{ma-mo-92}:
	\[
		\Set^{\mathbf{C}^{op}}/X\cong \Set^{(\int_{\mathbf{C}} X)^{op}}.
	\]
	
	\begin{proposition}\label{pr:algneriso}
		For any $P\in\Oper_{C}\Set$ there is a canonical isomorphism of categories:
		\[
			E\colon \Top^{\Omega^{op}_{C}} / NP \cong \Top^{(\Omega 	P)^{op}}\rcolon M.
		\]	
		Furthermore, one has an equality of functors $N_{P}=E\circ N(P,-)$.
	\end{proposition}
	\begin{proof}
	By definition, an object in the slice category $\Top^{\Omega^{op}_{C}} / NP$ is a dendroidal space $Y\in \Top^{\Omega_{C}^{op}}$ equipped with a morphism of presheaves $\alpha\colon Y \to NP$. 
	We construct an explicit functor $E\colon \Top^{\Omega_{C}^{op}} / NP \to \Top^{(\Omega P)^{op}}$. 
	For any object $(Y, \alpha)$ in $\Top^{\Omega_{C}^{op}} / NP$, we define a presheaf $E(Y,\alpha)$ on $\Omega P$ by taking the set-theoretic fibers of the components of $\alpha$. 
	Specifically, for $\mathcal{T}=(T, \psi_{T}) \in \Omega P$, we set:
	\[
		E(Y,\alpha)(\mathcal{T}) := 	\alpha_T^{-1}(\psi_{T}) = \lb x \in 	Y(T) \mid \alpha_T(x) = \psi_{T} \rb.
	\]
	For a morphism $f_{*}\colon \mathcal{S}=(S,\psi_{S}) \to \mathcal{T}$ induced by $f\colon S\to T$, the map 
	\[
		E(Y,\alpha)(f_{*})\colon E(Y,\alpha)(\mathcal{T}) \to E(Y,\alpha)(\mathcal{S})
	\] 
	is simply the restriction of the map $Y(f)\colon Y(T) \to Y(S)$. 
	This is well-defined because naturality of $\alpha$ guarantees that 
	\[
		\alpha_S\bigl(Y(f)(x)\bigr)=NP(f)\bigl(\alpha_T(x)\bigr)=NP(f)(\psi_T)=\psi_S,
	\]
	for $x$ with $\alpha_T(x)=\psi_T$, the last equality following from \eqref{cd_mophism_omegaP}. 
	Hence the
	restriction lands in the fibre over $\psi_S$.
	
	On morphisms, let $h\colon (Y,\alpha)\to (Y',\alpha')$ be a map over $NP$, so
	that $\alpha'h=\alpha$. 
	Then $h_T\bigl(\alpha_T^{-1}(\psi)\bigr)\subseteq
	(\alpha'_T)^{-1}(\psi)$ for every $\psi\in NP(T)$, and we let
	$E(h)_{\mathcal{T}}$ be the restriction of $h_T$ to the fibre over $\psi_T$.
	For a morphism $\beta\colon F\to F'$ in $\Top^{(\Omega P)^{op}}$ we let
	$M(\beta)_T:=\coprod_{\psi\in NP(T)}\beta_{(T,\psi)}$, a map over $NP$ by the
	construction of the structure maps $\alpha_T$. 
	Since $NP$ takes discrete
	values, a map over $NP$ is precisely a family of maps of the fibres, so $E$
	and $M$ are mutually inverse on morphisms as well.
	
	We construct the inverse functor $M\colon \Top^{(\Omega P)^{op}} \to \Top^{\Omega^{op}_{C}} / NP$. 
	Given a presheaf $F \in \Top^{(\Omega P)^{op}}$, we define a dendroidal space $Y = M(F)\in \Top^{\Omega_{C}^{op}}$ by taking the coproduct over all possible $P$-decorations of a tree $T$:
	\[
		Y(T) := \coprod_{\psi \in NP(T)} F(T, 	\psi).
	\]
	The map $\alpha_T\colon Y(T) \to NP(T)$ is defined as the canonical projection that sends the component $F(\mathcal{T})$ to the element $\psi$. 
	For any tree morphism $f\colon S \to T$, the map $Y(f)\colon Y(T)\to Y(S)$ is defined on the component $F(T,\psi)$ as
	\[
		F(f_*)\colon F(T,\psi)\longrightarrow 	F\bigl(S,NP(f)(\psi)\bigr),
	\]
	where $f_*\colon \bigl(S,NP(f)(\psi)\bigr)\to (T,\psi)$ is the morphism of
	$\Omega P$ determined by $f$ and $\psi$ (see~\S\ref{ssec:ptrees}). 
	Functoriality
	of $M(F)$ follows from that of $F$ and of $NP$.
	
	We verify that the functors $E$ and $M$ are strictly inverse to each other. 
	For a presheaf $X \in \mathbf{Top}^{(\Omega P)^{op}}$, the space $M(X)(T) = \coprod_{\psi \in NP(T)} X(T, \psi)$ is equipped with the natural projection $\alpha_T$ onto $NP(T)$. 
	The set-theoretic fiber of $\alpha_T$ over $\psi$ is exactly the component $X(T,\psi)$, hence $E(M(X)) = X$. 
	Conversely, for any object $(Y, \alpha)$ in the slice category $\Top^{\Omega_C^{op}}/NP$, the space $NP(T)$ is discrete, so that $\alpha_T$ is a continuous map into a
	discrete space; its fibres are therefore clopen, and a subset $U\subseteq Y(T)$
	is open if and only if $U\cap\alpha_T^{-1}(\psi)$ is open for every $\psi$.
	Thus $Y(T)$ decomposes as the topological coproduct of its fibres 
	$Y(T) = \coprod_{\psi \in NP(T)} \alpha_T^{-1}(\psi)$. 
	Because $E(Y)(\mathcal{T}) = \alpha_T^{-1}(\psi_{T})$, applying the functor $M$ to $E(Y)$ reconstructs this exact coproduct and its projection $\alpha_T$, yielding $M(E(Y, \alpha)) = (Y, \alpha)$. 
	This establishes the canonical isomorphism of categories.		
	
	We prove the identity between functors from the claim of this proposition.
	By the construction, for a $P$-algebra $A$ one has 
	\[
		N(P,A)(T) = \coprod_{\psi \in NP(T)} \prod_{l \in L(T)} A(l),
	\]
	where the structure map $\alpha_T \colon N(P, A)(T) \to NP(T)$ is the projection from the coproduct 
	\[
		\coprod_{\psi \in NP(T)} \prod_{l \in L(T)} A(l)
	\] 
	to its index set $NP(T)$, which maps the component indexed by $\psi$ to the element $\psi$. 
	Applying the fiber functor $E$ one obtains the leaf-decoration space over $\mathcal{T}$:
	\[
		E(N(P, A))(\mathcal{T}) = 	\alpha_T^{-1}(\psi_{T}) = \prod_{l \in L(T)} A(l) = N_P A(\mathcal{T}).
	\]
	It remains to verify that this identification respects the presheaf restriction maps. 
	Any morphism $f_{*}\colon \mathcal{S} \to \mathcal{T}$ in $\Omega P$ is induced by a tree morphism $f\colon S \to T$ and  is compatible with the respective $P$-decorations. 
	The induced restriction map $E(N(P, A))(f)$ is simply the restriction of the map $N(P, A)(f)$ to the fiber over the discrete set $\psi$. 
	We verify this matches $N_P A(f)$ on the generators of $\Omega P$:
	\begin{itemize}
		\item \textbf{Inner faces:} Contracting an inner edge does not change the set of leaves, meaning $l(S) = l(T)$. 
		The restriction map on the leaf decoration space $\prod_{l \in L(T)} A(l)$ is simply the identity map. 
		This matches the action of $N_P A$, which sends inner faces to the identity to enforce associativity.
		
		\item \textbf{Outer faces:} Pruning an external corolla $v$ (decorated by the respective operation $p \in P(v)$) removes its input leaves and replaces them with a single new leaf (the respective root edge). 
		By the definition of the dendroidal nerve $N(P,A)$, the restriction map evaluates the operation $p$ on the corresponding leaf variables using the $P$-algebra structure of $A$. 
		This coincides with the structure maps of $N_P A$.
		
		\item \textbf{Root outer faces:} The inclusion $T_{\geqslant e}\subseteq T$ of a
		branch restricts on $\prod_{l\in L(T)}A(l)$ to the projection onto
		$\prod_{l\in L(T_{\geqslant e})}A(l)$, which is the corresponding structure
		map of $N_PA$.
		\item \textbf{Degeneracies:} Contracting a unary identity vertex leaves the
		respective leaf set unchanged, and the restriction map on the fibres is the
		identity, matching the strict unit condition of $N_PA$.

		\end{itemize}
		Since the maps $\alpha_{T}$ agree on all generating morphisms, we have the equality of presheaves $E(N(P, A)) = N_P A$.
		The proof is complete.
	\end{proof}	
		
	\subsection{Homotopy nerve of homotopy algebras}
	
	\begin{definition}\label{def:nhp}
		Define the \emph{homotopy nerve of a homotopy $P$-algebra} as the following functor
		\[
			N^{h}_{P}\colon 	\Alg(WP,\Top)\to 	\Top^{(\underline{\Omega} P)^{op}}_{strict},
		\]
		given by $N^{h}_{P}(A):=N_{P}(A)$ on any homotopy algebra object $A$ by~\eqref{eq:algner}, and by adjunction define 
		\[
			\underline{\Omega} P(\mathcal{S},\mathcal{T})\times 
			\prod_{l'\in L(T)} A(l')\to 
			\prod_{l\in L(S)} A(l),\ 
			(f,\mu)\mapsto \phi_{\mathcal{T}[f(l)]}(\mu_{l},-).
		\]
	\end{definition}
	
	\begin{lemma}
		The functor $N^{h}_{P}$ is well defined.
	\end{lemma}
	\begin{proof}
		We prove that the definition of $N^{h}_{P}$ is compatible with the composition in Definition~\ref{def:thick}.
		The units of the composition correspond to $\phi_{\eta}$.
		The claim on compatibility in terms of~\eqref{eq:comp} is
		\[
		\phi_{\mathcal{T}[ gf(l)]}((\nu*\mu)_{l},-)=
		\phi_{\mathcal{U}[f(l)]}(\mu_{l},-)\circ\prod\limits_{m} \phi_{\mathcal{T}[g(m)]}(\nu_{m},-).
		\]
		By restriction to the leaves (see Lemma~\ref{lm:leafcomp}), it is enough to check equality of two elements for $WP$ (i.e. $P$-trees with metric on their inner edges) (on the right one has the infinitesimal composition of the operad $WP$):
		\begin{equation}\label{eq:wpcomp}
			(\mathcal{T}[gf(l)],(\nu*\mu)_{l})=
			(\mathcal{U}[f(l)],\mu_{l})\circ_{WP}
			(\mathcal{T}[g(m)],\nu_m{})_{m\in L(\mathcal{U}[f(l)])}.
		\end{equation}
		This is proved by case study by the definition of $\nu*\mu$, see~\eqref{eq:comp}, and properties of $W$-construction.
		The RHS is the tree $\mathcal{U}[f(l)]$ with each $\mathcal{T}[g(m)]$ grafted at the leaf $m$.
		On the inner edges of $\mathcal{U}[f(l)]$ the lengths are $\mu$, on those of $\mathcal{T}[g(m)]$ these are $\nu$, and $1$ on each grafting edge.
		On the left-hand side of~(\ref{eq:wpcomp}), by Lemma~\ref{lm:branches} we have the same tree with edge labels given by~\eqref{eq:comp}.
		The only possible discrepancies are the type $(ii)$ edges, of length $0$, and the edges with several preimages
		under $g$, whose length is the maximum of the respective lengths. 
		The latter is the unit relation of
		$W P$: the vertices deleted by $g$ are decorated by identities of $P$, and deleting
		such a valence $2$ vertex replaces the two adjacent lengths by their maximum.
		Since a length $0$ edge is identified with its contraction in $WP$, contracting all type $(ii)$ edges collapses each expansion $g(C_{w})$ to one vertex, whose decoration is given by $\psi_{\mathcal{U}}(w)$ by~\eqref{cd_mophism_omegaP}. 
		This proves~\eqref{eq:wpcomp}. 
	\end{proof}

	\begin{lemma}\label{lm:ssegal}
		For any strict Segal presheaf $X\colon  (\underline{\Omega} P)^{op}\to \Top$ and any $(f,\mu)\in\underline{\Omega}P(\mathcal{S},\mathcal{T})$ the $l$-th component of $X(f,\mu)$, $l\in L(S)$, equals
		\[
			X(\iota_{\mathcal{S},l},\varnothing)\circ 	X(f,\mu)= X\bigl(\iota_{\mathcal{T},f(l)},\mu_l\bigr)\colon X(\mathcal{T})\to X(\eta_l).
		\]
	\end{lemma}
	\begin{proof}
		By the strict Segal condition the map
		\[
		X(f,\mu)\colon 
		\begin{tikzcd}
		X(\mathcal{T})\arrow{r} &
		X(\mathcal S)\arrow{r}{\cong} &
		\prod\limits_{l'\in l(S)}X(\eta_{l'})
		\end{tikzcd}
		\] 
		is determined by its components
		$X(\iota_{l'},\varnothing)\circ X(f,\mu)$.
		Then by Lemma~\ref{lm:leafcomp} one has
		\[
		X(\iota_l,\varnothing)\circ 		X(f,\mu)=X\bigl((f,\mu)\circ(\iota_l,\varnothing)\bigr)=
		X\bigl(\iota_{\mathcal{T}, 		f(l)},\mu_l\bigr).
		\]
		This proves the claim.
	\end{proof}

	\begin{remark}\label{rm:etacomponents}
		The set $\Omega P(\eta_c,\mathcal{T})$ consists of the inclusions $\iota_{\mathcal{T},e}$
		of the edges $e$ of colour $c$, and $M_{\iota_{\mathcal{T},e}}\cong I^{ie(T[e])}$:
		a morphism $\eta_c\to\mathcal{T}$ in $\underline{\Omega}P$ is a branch of $\mathcal{T}$
		together with a metric on its inner edges. For $e=r(T)$ one obtains the root component
		$I^{ie(T)}$.
	\end{remark}

	\begin{theorem}\label{thm:halgdiag}
		For any unital $C$-colored operad $P$ in sets $N^{h}_{P}$ is an isomorphism of categories.
	\end{theorem}
	\begin{proof}
		By the construction, $N^{h}_{P}(A)$ is a strict Segal presheaf for any homotopy $P$-algebra $A$.
		Conversely, define the inverse functor $R$ to $N^{h}_{P}$ by letting $R(X)(c):=X(\eta_{c})$ on any presheaf $X\colon (\underline{\Omega} P)^{op}\to\Top$.
		The root component of  $\underline{\Omega}P(\eta_{c},\mathcal{T})$ is $I^{ie(T)}$.
		Furthermore, by Lemma~\ref{lm:ssegal} it is enough to know $X$ on
		morphisms with domain $\eta_l$.
		Then define the structure map for $A=R(X)$ by
		\[
			\phi_{\mathcal{T}}\colon
			I^{ie(T)}\times \prod\limits_{l\in L(T)} A(l)\cong		
			I^{ie(T)}\times X(\mathcal{T})\to 
			X(\eta_{r(T)})=A(r(T)),
		\]
		using strictness of $X$.
		We check the $WP$-axioms by looking at compositions of morphisms in $\underline{\Omega} P$.
		For $t_{e}=0$ consider 
		\[
			\begin{tikzcd}
				\eta_{r}\arrow{r}{(f,\lambda)} &
				\mathcal{T}/e\arrow{r}{(\partial_{e},\varnothing)} &
				\mathcal{T}			
			\end{tikzcd}
		\] 
		for some length function $\lambda$.
		In $\underline{\Omega} P$ this chain composes by type $(ii)$ (length $0$) at edge $e$ and by type $(i)$ at all other edges (whose lengths are given by the respective values of $\lambda$). 
		Therefore, the corresponding relation for the presheaf $X$ (letting $F=\lb t_{e}=0\rb$) is
		\[
			\phi_{\mathcal{T}}|_{F}=
			\phi_{\mathcal{T}/e}\circ X(\partial_{e},\varnothing).
		\]
		Notice $X(\partial_{e},\varnothing)=\Id$ by Lemma~\ref{lm:ssegal}, and that the leaves of $\mathcal{T}$, $\mathcal{T}/e$ are the same.
		Hence the necessary $W$-identity follows.
		
		For $t_{e}=1$ write $\mathcal{T}=\mathcal{B}\circ \mathcal{U}_{e}$ as a grafting of one branch (the general case then follows by induction).
		The chain
		\[
		\begin{tikzcd}
			\eta_{r}\arrow{r}{(f,\lambda')} &
			\mathcal{B}\arrow{r}{(g,\lambda'')} &
			\mathcal{T},
		\end{tikzcd}
		\] 
		composes at the grafting edge $e$ by type $(iii)$ (length $1$), at $\mathcal{B}$ by type $(i)$ ($\lambda'$), at the branch $\mathcal{U}_{e}$ by type $(iv)$ ($\lambda''$).
		The respective composition for $X$ (writing $F=\lb t_{e}=1\rb$) is
		\[
			\phi_{\mathcal{T}}|_{F}=
			\phi_{\mathcal{B}}\circ \phi_{\mathcal{U}_{e}}.
		\]

		Consider a degeneracy $\sigma_{e}\colon \mathcal{T}'\to \mathcal{T}$. 
		It removes a valence $2$ vertex $v$ (incident to $e',e''$) and produces the edge $e$.
		The degeneracy vertex is decorated by an identity in $P$. 
		Consider the composition
		\[
		\begin{tikzcd}
			\eta_{r}\arrow{r}{(f,\lambda)} &
			\mathcal{T}'\arrow{r}{(\sigma_{e},\varnothing)} &
			\mathcal{T}.
		\end{tikzcd}		
		\]
		It is given by type $(i)$ (length $t_{e}=\max\lb\lambda(e'),\lambda(e'')\rb$) at $e$, as well as at all other edges (whose lengths are given by $\lambda$).
		The respective identity for $X$ (writing $F=\lb t_{e}=\max\lb\lambda(e'),\lambda(e'')\rb\rb$) is
		\[
			\phi_{\mathcal{T}}|_{F}=
			\phi_{\mathcal{T}'},
		\]
		i.e. the unit axiom.
		Hence, $A=R(X)$ is a homotopy $P$-algebra.

		Clearly, the functors $N^{h}_{P}$ and $R$ are inverse to each other on all objects.
		Consider the natural morphism
		\begin{equation}\label{eq:embmor}
			\Alg_{WP}(A,B)\to
			\Nat(N^{h}_{P}A,N^{h}_{P}B).
		\end{equation}
		Given any natural transformation $\alpha\colon N^{h}_{P}A\to N^{h}_{P}B$, write for a $P$-tree $\mathcal{T}$
		\[
			\alpha_{\mathcal{T}}=\prod\limits_{l\in L(T)} \alpha_{l},
		\] 
		by the Segal condition.
		Hence, $\alpha$ is uniquely determined by $\alpha_{\eta_{c}}$, and~\eqref{eq:embmor} is injective.
		By Lemma~\ref{lm:leafcomp}, it is enough to determine a functor on morphisms having leaf domain only.
		By naturality at 
		\[
			(\iota_{\mathcal{T}},\mu)\colon \eta_{r(T)}\to \mathcal{T},
		\]
		one has
		\[
			\alpha_{\eta_{r(T)}}\circ
			\phi^{A}_{\mathcal{T}}(\mu)=
			\phi^{B}_{\mathcal{T}}(\mu)\circ
			\prod_{l\in L(T)}\alpha_{l}.
		\]
		By Lemma~\ref{lm:leafcomp}, every element of $WP$ has a form $(\mathcal{T},\mu)$.
		Then this implies that $\alpha=(\alpha_{\eta_{c}})$ commutes with the homotopy algebra structure.
		Hence,~\eqref{eq:embmor} is surjective.
		This proves the claim on isomorphism of categories.
	\end{proof}
	
	\begin{example}\label{ex:diag}
		Let $\Cc$ be a small category and $P=\Diag_{\Cc}$ the associated unary operad, so that
		homotopy $P$-algebras are homotopy coherent diagrams over $\Cc$ \cite{be-mo-06}. 
		A $P$-tree is a chain $s=(c_0\to\dots\to c_n)$ in $\Cc$, drawn as a linear tree with
		$n$ vertices, root edge $c_n$ and the single leaf $c_0$; in particular the strict
		Segal condition reads $X(\mathcal{T}_{s})\cong X(\eta_{c_{0}})$.
		The unique branch of the
		root inclusion $\iota_{\mathcal{T}}$ is the whole tree, so the root component of $\underline\Omega P(\eta_{c_n},\mathcal T_s)$ is $I^{n-1}$.
		For any homotopy algebra $A$ over $P$, the corresponding morphism $N^h_PA \iota_{\mathcal{T}}$ is the higher homotopy
		$F(01\dots n)\colon I^{\,n-1}\times F(c_0)\to F(c_n)$ of the respective homotopy coherent diagram. 
		The facet $t_e=0$ contracts
		an inner edge, i.e. composes two adjacent arrows of $s$, giving
		$F(01\dots\hat{i}\dots n)$; the facet $t_e=1$ splits the chain, yielding the composition
		$F(i\dots n)\circ F(01\dots i)$. 
		One thus recovers the cubes of Vogt and
		Cordier--Porter~\cite{co-po-86} on the nose.
	\end{example}
	
	\subsection{Horns of homotopy algebras revisited}\label{ssec:comphorns}	
	We compare horns for a homotopy $P$-algebra $A$ (definition~\ref{defn:alghorn}) and its homotopy dendroidal nerve $N_P^h A$. 
	Let $A \in \mathbf{Alg}(W P, \Top)$ be a homotopy $P$-algebra, and let $\mathcal{T} = (T, \psi_{T})$ be a $P$-tree with root edge $\eta_r$. 
	
	\begin{proposition}\label{pr:horncorr}
		An operadic horn $\Lambda_x(A, \mathcal{T}): \Lambda_x(T) \to \mathrm{End}(UA)$ is canonically equivalent to the restriction of $N^{h}_{P} A$ to the $x$-box of $\underline{\Omega} P(\eta_{r(T)},\mathcal{T})$.
	\end{proposition}
	\begin{proof}
		The root component of $\underline{\Omega} P(\eta_{r(T)},\mathcal{T})$ is $I^{ie(T)}$, and the facets $t_{e}=0,1$ correspond to $\theta_{\mathcal{T}/e}$ ($t_{e}=0$ for $x=e$) and $\theta_{\mathcal{B}}\circ\theta_{\mathcal{U}}$ ($t_{e}=1$ for $x=v$), respectively.
		Then two notions of an $x$-horn are canonically equivalent to each other as unions of facets of the respective cubes.
	\end{proof}

	\section{Comparison with (homotopy) operads}\label{sec:compoper}

	Throughout this section $\mathcal{O}_C$-trees are planar, and we write $ie^{O}(\mathcal{T})$
	for the set of inner edges of an $\mathcal{O}_C$-tree $\mathcal{T}$, keeping $ie(T)$ for the
	inner edges of a $C$-coloured tree $T$. 
	Recall from the proof of Theorem~\ref{thm:halgdiag} that there is the functor
	\[
	R=\bigl(N^h_{\mathcal{O}_C}\bigr)^{-1}\colon
	\Top^{\underline{\Omega}{\mathcal{O}_C^{op}}}_{strict}
	\rightarrow \Alg(W\mathcal{O}_C,\Top).
	\]
	
	\subsection{Embedding for strict operads}\label{ssec:strictop}

	For the $\N$-colored operad $\mathcal{O}$ of operads (in sets) the respective algebras form $1$-colored operads~\cite[\S2.4]{clrrw-22}:
	\[
		\Oper^{1}_{\Top}\cong 	\Alg(\mathcal{O},\Top).
	\]
	More generally, $C$-colored operads are algebras over the operad $\mathcal{O}_{C}$ in sets~\cite[\S 1.5.6, \S 1.5.7]{be-mo-06},~\cite[\S 3]{gu-vo-12}.	
	The set of colors in $\mathcal{O}_{C}$ is given by all $C$-colored corollas.
	In the particular case $P=\mathcal{O}_{C}$ (see~\cite{clrrw-22}), there is an embedding
	\[
		\Psi\colon 	\Alg(\mathcal{O}_{C},\Top)\subset \Top^{\Omega^{op}_{C}}_{strict},
	\]
	to strict Segal presheaves, see~\cite{clrrw-22}.
	The presheaf $\Psi(P)$ sends $\eta$ to a point, and $T$ to $\prod_{v\in iv(T)} P(v)$.
	On a morphism $g\colon S\to T$, 
	\[
	\Psi(P)(g)\colon \prod_{v\in iv(T)} P(v)\to \prod_{w\in iv(S)} P(w)
	\]
	is the natural composition.
	(Notice that outer faces contribute by the coordinatewise projections.)

	\subsection{Embedding for homotopy operads}\label{ssec:homop}
	Recall the definition of the explosion category $\widetilde{\Omega}$ of $\Omega_C$ from~\cite[Def.~B.1]{clrrw-22}.
	For a morphism $g\colon S\to T$ in $\Omega_C$, let $\Path_\Omega(S,T)_g$ be the poset of
	factorizations of $g$ modulo identity morphisms, ordered by refinement, and let
	$K_g:=|\Path_\Omega(S,T)_g|$. The topological category $\widetilde\Omega$ has the
	objects of $\Omega_C$ and
	\[
	\widetilde{\Omega}(S,T):=
	\coprod_{g\in\Omega_C(S,T)}K_g ,
	\]
	with composition given by the concatenation of factorizations. 
	
	In~\cite[Thm.~B.6]{clrrw-22} the full embedding
	\[
		\Psi^{h}\colon\Alg(W\mathcal{O}_C,\Top)\subset\Top^{\widetilde\Omega^{op}}_{strict}
	\]
	is constructed. 
	By~\cite[(B.1)]{clrrw-22} the embedding $\Psi^h$ acts at a vertex $v\in iv(S)$ through
	the map 
	\[
		\alpha_v\colon K_g\times\prod_{w\in iv(g(\mathcal{C}_v))}A(w)\to A(v).
	\]
	Since $A$ is a $W\mathcal{O}_C$-algebra, $\alpha_v$ is adjoint to a map
	\[
		\widehat{\alpha}_v\colon K_g\rightarrow
		\eps^{-1}\bigl(g(\mathcal{C}_v)\bigr)\subseteq W\mathcal{O}_C .
	\]

	\subsection{Factorizations of dendroidal morphisms, and auxiliary $\mathcal{O}_{C}$-trees}\label{ssec:facaux}
	For any $g\colon S\to T$ let
	$P_g:=iv(T)\setminus iv(g(S))$ be the set of pruned vertices.

	\begin{lemma}\label{lem:bookkeeping}
		\begin{enumerate}[label=(\roman*)]
			\item For any $g\colon S\to T$, one has 
			\[
			iv(T)=\Bigl(\coprod_{v\in iv(S)} iv\bigl(g(C_v)\bigr)\Bigr)\amalg P_{g}.
			\]
			\item For any $\mathcal{O}_C$-tree $\mathcal{T}$ one has $L(\mathcal{T})=iv(\eps(\mathcal{T}))$.
		\end{enumerate}
	\end{lemma}
	\begin{proof}
		We prove the first claim. 
		By the definition of $P_g$ one has $iv(T)=iv(g(S))\amalg P_g$, the
		complement being the set of vertices pruned by the outer face $g(S)\subseteq T$. Since
		$S=\bigcup_{v\in iv(S)}C_v$, the expansions cover,
		$g(S)=\bigcup_{v}g(C_v)$. 
		They are disjoint on vertices: for $v\neq v'$ the corollas $C_v$, $C_{v'}$ share at most an edge, so $g(C_v)$ and $g(C_{v'})$ share at most an edge as well. 
		Finally $iv(g(C_v))=\varnothing$ precisely when $g$ contracts the unary vertex
		$v$, i.e. $g(C_v)=\eta$.
		
		The second claim follows by the induction on the number of vertices of $\mathcal{T}$, using that $\eps$
		substitutes the decorations of the vertices of $\mathcal{T}$ into one another.
		A substitution identifies the leaves of the result with the vertices of the substituted trees.
	\end{proof}

	Let $\sigma=(g_1,\dots,g_n)$ be any factorization of $g$, and let $v\in iv(S)$.
	For $0\leqslant i\leqslant n$ let $T_i$ denote the target of $g_i\cdots g_1$,
	with $T_0:=S$ and $T_n=T$, and let
	\[
		V_i:= iv\bigl(g_i\cdots g_1(C_v)\bigr)\subseteq iv(T_i)
	\]
	be the set of vertices of the subtree of $T_i$ into which $v$ has been
	expanded $i$ times.
	Thus $V_0=\{v\}$ and $V_n=iv\bigl(g(C_v)\bigr)$. 
	One has
	\[
		V_{i+1}=\coprod_{w\in V_i} iv\bigl(g_{i+1}(C_w)\bigr).
	\]

	\begin{definition}\label{def:Gv}	
		For any factorization $\sigma=(g_1,\dots,g_n)$ of $g$, and any $v\in iv(S)$, define the $\mathcal{O}_C$-tree $\mathcal G_v(\sigma)\in\Omega \mathcal{O}_C$ as follows:
		\begin{itemize}
			\item \textbf{edges:} one edge for every pair $(w,i)$ with $w\in V_i$,
			$0\leqslant i\leqslant n$, coloured by the corolla $C_w$; the edge of
			$(v,0)$ is the root, the edges of $(w,n)$, $w\in V_n$, are the leaves;
			\item \textbf{vertices:} one vertex for every pair $(w,i)$ with $w\in V_i$,
			$0\leqslant i<n$, whose output edge is the edge of $(w,i)$ and whose input
			edges are the edges of $(w',i+1)$ for $w'\in iv(g_{i+1}(C_w))\subseteq
			V_{i+1}$;
			\item \textbf{decoration:} the vertex at $(w,i)$ is decorated by the tree
			$g_{i+1}(C_w)$, viewed as an operation of $\mathcal{O}_C$ with inputs
			$iv(g_{i+1}(C_w))$ and output $C_w$.
		\end{itemize}
	\end{definition}	

	\begin{remark}
		If $g_{i+1}$ does not expand $w$, the vertex at $(w,i)$ is decorated by the
		corolla $C_w$, i.e. by the unit of $\mathcal{O}_C$ at the colour $C_w$, and is
		deleted. 
		If $g_{i+1}$ contracts the unary vertex $w$, then $iv(g_{i+1}(C_w))=\varnothing$ and the vertex at $(w,i)$ is decorated by the trivial tree $\eta\in\mathcal{O}_C(C_w)$. This is a nullary operation of $\mathcal{O}_C$
		(not an identity), and the vertex is kept. 
		It has no input edge, so the edge of
		$(w,i)$ is not a leaf of $\mathcal{G}_v$. In particular $\mathcal{G}_v$ has no
		identity-decorated vertex, so it depends only on the class of
		$(g_1,\dots,g_n)$ modulo identity morphisms.
	\end{remark}
	
	\subsection{A comparison of two posets}\label{ssec:comppos}
	
	For $v\in iv(S)$ let $\mathcal{P}_{g,v}$ be the set of $\mathcal{O}_C$-trees $\mathcal{G}$ with
	$\eps(\mathcal{G})=g(C_v)$, partially ordered by edge collapse.
	That is,
	$\mathcal{G}'\leq\mathcal{G}$ iff $\mathcal{G}'$ is obtained from $\mathcal{G}$ by collapsing a set of inner edges, each collapse composing in $\mathcal{O}_C$ the decorations of the two vertices of the collapsed edge. 
	Such a collapse leaves the total composite
	unchanged, so $\mathcal{P}_{g,v}$ is the poset corresponding to $g(C_{v})$ from~\cite[Rem.~A.10]{clrrw-22}.
	By~\cite[Rem.~A.10]{clrrw-22}, the barycentric coordinates of the realization equate the lengths of the inner edges, i.e.
	\[
		\bigl|\mathcal{P}_{g,v}\bigr|\ \cong\ 	\eps^{-1}\bigl(g(C_v)\bigr)\subseteq W\mathcal{O}_C.
	\]

	\begin{remark}\label{rem:egeq}
		Composing the decorations from the leaves to the root gives
		$\eps(\mathcal{G}_v(\sigma))=g(C_v)$, so $\mathcal{G}_v(\sigma)\in P_{g,v}$.
		By Lemma~\ref{lem:bookkeeping} $(ii)$,
		\[
			L\bigl(\mathcal{G}_v(\sigma)\bigr)=
			iv\bigl(\eps(\mathcal{G}_v(\sigma))\bigr)=
			iv\bigl(g(C_v)\bigr).
		\]
	\end{remark}		
	
	\begin{lemma}\label{lm:pi}
		Let $g\colon S\to T$ in $\Omega_C$ and $v\in iv(S)$. 
		There is the morphism of posets
		\[
			\Path_\Omega(S,T)_g\longrightarrow\mathcal P_{g,v},\ 
			\sigma\mapsto\mathcal G_v(\sigma).
		\]
		Its realization
		\[
		\pi^v_g\colon K_g\rightarrow \lvert P_{g,v}\rvert\cong
		\eps^{-1}\bigl(g(C_v)\bigr)\subseteq W\mathcal{O}_C
		\]
		at a vertex $\sigma$ of $K_g$ is $\mathcal{G}_v(\sigma)$, i.e. a metric $\mathcal{O}_{C}$-tree with all inner edges of length $1$. 
	\end{lemma}
	\begin{proof}
		Refining a factorization replaces a stage $g_{i+1}$ by a composite
		$g'\circ g''$, which subdivides the vertices of $\mathcal{G}_v$ at that stage;
		conversely, composing two adjacent stages collapses the corresponding inner
		edges of $\mathcal{G}_v$ and composes their decorations in $\mathcal{O}_C$. 
		Hence the assignment is monotone. 
		Under the identification
		$\lvert P_{g,v}\rvert\cong
		\eps^{-1}(g(C_v))$ of~\cite[Rem.~A.10]{clrrw-22}, a point of the simplex spanned by a chain
		$\sigma_0\leqslant\dots\leqslant\sigma_k$ with barycentric coordinates
		$(t_0,\dots,t_k)$ has lengths
		\begin{equation}\label{eq:affine}
			\len(a)=\sum_{i\colon a\in ie^{O}(G_v(\sigma_i))}t_i,
			\ a\in ie^{O}\bigl(\mathcal{G}_v(\sigma_k)\bigr),
		\end{equation}
		which is affine in the $t_i$ and equals $1$ at a vertex.
	\end{proof}
	
	\subsection{A comparison of two embeddings}\label{ssec:compemb}
	
	\begin{definition}\label{def:hoflat}
		Define the \emph{homotopy flattening functor} to be the composite
		\[
		\mathrm{Flat}^h:=\Psi^{h}\circ R\colon
		\Top^{\underline{\Omega}\mathcal{O}_{C}^{op}}_{\mathrm{strict}}
		\longrightarrow
		\Top^{\widetilde{\Omega}^{op}}_{strict}.
		\]
	\end{definition}
	
	Since $R$ is
	an isomorphism of categories by
	Theorem~\ref{thm:halgdiag} and $\Psi^{h}$ is a full embedding by
	\cite[Thm.~B.6]{clrrw-22}, $\mathrm{Flat}^h$ is a full embedding as well. 
	The next proposition identifies it explicitly and proves the comparison with
	$\Psi^{h}$.
	For $T\in\Omega_{C}$ let $C(T)\in\Omega\mathcal{O}_{C}$ be the single corolla decorated by $T$.
	Its leaves are $iv(T)$ by Lemma~\ref{lem:bookkeeping} $(ii)$.
	Let 
	\[
	\pi_g:=
	(\pi^v_g)_{v\in iv(S)}\colon K_{g}\to \prod \lvert P_{g,v}\rvert
	\] 
	for any $g\in \Omega_C (S,T)$.
	
	\begin{theorem}\label{thm:comparison}
		The following diagram commutes:
			\[
		\begin{tikzcd}[column sep=huge, row sep=large]
			\Alg(W\mathcal{O}_C,\Top)\arrow{r}{N^h_{\mathcal{O}_C}}\arrow[swap]{dr}{\Psi^{h}}
			& \Top^{(\underline\Omega \mathcal{O}_C)^{op}}_{\mathrm{strict}}\arrow{d}{\Flat^h}\\
			& \Top^{\widetilde\Omega^{op}}_{\mathrm{strict}}
		\end{tikzcd}
		\]
		Moreover, on objects one has $\mathrm{Flat}^h(X)(T)=X(C(T))$.
		The map $\Flat^h(X)(\sigma)$ for $\sigma\in K_g$, $g\colon S\to T$ decomposes as follows:
		\begin{equation}\label{eq:flath}
			\begin{tikzcd}[row sep=large, column sep=large]
			X\bigl(C(T)\bigr) \arrow{r}{\cong} &
			X(\mathcal{T}_g)\times\prod_{w\in P_g}X(\eta_{C_w})\arrow{r}{pr} &
			X(\mathcal{T}_g)\arrow{r}{X(j_g,\pi_g(\sigma))} &
			X\bigl(C(S)\bigr).
			\end{tikzcd}
		\end{equation}
		Here, $\mathcal{T}_g=\mathcal{T}_g(\sigma)\in\Omega\mathcal{O}_C$ is the $\mathcal{O}_C$-tree whose root vertex is decorated
		by $S$ and which has $\mathcal{G}_v(\sigma)$ grafted at the leaf $v$ of $C(S)$.
		Then $j_g\colon C(S)\to \mathcal{T}_g$ is the inclusion of the root corolla, the first map
		is the strict Segal identification and the second one is the projection
		omitting the factors indexed by the pruned vertices.
	\end{theorem}
	\begin{proof}
		The claim on commutativity follows trivially from definition~\ref{def:hoflat}. 
		For the explicit description put
		$A:=R(X)$, so that $X=N^h_{\mathcal{O}_C}A$ and $X(\eta_{C_w})=A(w)$ for
		$w\in iv(T)$.
		By the strict Segal condition one has
		\[
		X(C(T))=
		\prod_{l\in L(C(T))}A(l)=
		\prod_{w\in iv(T)}A(w)=
		\Psi^{h}(A)(T),
		\] 
		which
		is the claim on objects and the object part of the diagram.
		
		Let $\sigma\in K_g$. By Remark~\ref{rem:egeq} and Lemma~\ref{lem:bookkeeping} $(i)$,
		\[
		L(\mathcal{T}_g)=
		\coprod_{v\in iv(S)} L\bigl(\mathcal{G}_v(\sigma)\bigr)=
		\coprod_{v\in iv(S)} iv\bigl(g(C_v)\bigr)= iv(T)\setminus P_{g},
		\]
		so the first two maps of \eqref{eq:flath} are exactly the projection away
		from the pruned vertices of \cite[(B.1)]{clrrw-22}. The pair
		$(j_g,\pi_g(\sigma))$ is a morphism of $\underline{\Omega}{\mathcal{O}_C}$, the branch of $\mathcal{T}_g$
		at the leaf $v$ of $C(S)$ being $\mathcal{G}_v(\sigma)$.
		By Lemma~\ref{lm:ssegal}
		its $v$-component is
		\[
		X\bigl(\iota_{\mathcal{T}_{g}[v]},\pi^v_g(\sigma)\bigr)=
		\varphi_{\pi^v_g(\sigma)}\colon
		\prod_{w\in iv(g(C_v))}A(w)\rightarrow A(v),
		\]
		i.e. the action of the $W\mathcal{O}_C$-operation $\pi^v_g(\sigma)$.
		It therefore suffices to verify in $\eps^{-1}\bigl(g(C_v)\bigr)$ the identity
		\begin{equation}\label{eq:opstar}
			\widehat{\alpha}_v(\sigma)=
			\pi^v_g(\sigma),\ 
			v\in iv(S).
		\end{equation}
	
		Both sides of~\eqref{eq:opstar} are realizations of maps of posets defined on
		$\Path_\Omega(S,T)_g$.
		The left-hand side by~\cite[App.~B]{clrrw-22}, where
		$\alpha_v$ is obtained by realizing the assignment sending a factorization to the
		corresponding decorated tree, a refinement of factorizations corresponding to a collapse
		of the edges at the added levels.
		For the right-hand side, Lemma~\ref{lm:pi} together with
        (\ref{eq:affine}) shows directly that $\pi_g^{v}$ is affine on each simplex of $K_g$. Hence, to verify (\ref{eq:opstar}), it is enough to check the equality at the vertices $\sigma=(g_1,\dots,g_n)$ of $K_g$.
 		All the respective inner edges have length $1$ on both sides by Lemma~\ref{lm:pi} and by~\cite[(B.1)]{clrrw-22}.
		Hence,~\eqref{eq:opstar} reduces to an equality of decorated $\mathcal{O}_C$-trees. 
		The tree underlying $\widehat{\alpha}_v(\sigma)$ is obtained from the restriction
		\[
		C_v\rightarrow g_1(C_v)\rightarrow\dots\rightarrow g(C_v)
		\subseteq T
		\]
		of $\sigma$ to the respective corolla by attaching at height $i+1$ one vertex for every vertex $w$ of
		$(g_i\cdots g_1)(C_v)$, decorated by its image $g_{i+1}(C_w)$.
		The respective levels decorated by identities being deleted. 
		The resulting $\mathcal{O}_{C}$-tree is precisely $\mathcal{G}_v(\sigma)$ of definition~\ref{def:Gv}. 
		(Notice that the deletion of the identity-decorated levels takes place at the
		level of factorizations, i.e. before taking the sum in~\eqref{eq:affine}. 
		In particular no relation of the $W$-construction is invoked on interpolated
		lengths.)
		This proves \eqref{eq:opstar}, hence the proposition.
	\end{proof}
	
	
	\vskip 1em
	
	\begin{figure}
		\begin{center}    
			\begin{tikzpicture}[scale=0.8, every node/.style={font=\small}]
				
				\coordinate (v1) at (0,1.2);
				\coordinate (v2) at (1,0.2);
				
				\coordinate (a1) at (-1.0,2.5);
				\coordinate (a2) at (0.4,2.5);
				
				\coordinate (b1) at (1,2.5);
				\coordinate (b2) at (2,2.5);
				\coordinate (b3) at (3,2.5);
				
				\coordinate (r) at (1,-1.2);
				
				\draw[thick] (a1) -- (v1);
				\draw[thick] (a2) -- (v1);
				\draw[thick] (v1) -- (v2);
				
				\draw[thick] (b1) -- (v2);
				\draw[thick] (b2) -- (v2);
				\draw[thick] (b3) -- (v2);
				
				\draw[thick] (v2) -- (r);
				
				\fill (v1) circle (5pt);
				\fill (v2) circle (5pt);
				
				\coordinate (A1) at (-1,0.6);
				
				\coordinate (v) at ($(0,0) + (A1)$);
				\coordinate (u) at ($(-0.25,0.5) + (A1)$);
				
				\coordinate (l1) at ($(-0.5,1) + (A1)$);
				\coordinate (l2) at ($(0.5,1) + (A1)$);
				\coordinate (r)  at ($(0,-0.7) + (A1)$);
				
				\draw[thick] (l1) -- (u);
				\draw[thick] (u) -- (v);
				\draw[thick] (l2) -- (v);
				\draw[thick] (v) -- (r);
				\draw[dotted, thick] (v) -- (v1);
				
				\fill (v) circle (4pt);
				\fill (u) circle (4pt);
				
				
				\coordinate (A2) at (2.5,-0.2);
				
				\coordinate (w2) at ($(0.25,0.5) + (A2)$);
				\coordinate (u2) at ($(0,0) + (A2)$);
				
				\coordinate (l02) at ($(-0.5,1) + (A2)$);
				\coordinate (l12) at ($(0,1) + (A2)$);
				\coordinate (l22) at ($(0.5,1) + (A2)$);
				\coordinate (r2)  at ($(0,-0.7) + (A2)$);
				
				\draw[thick] (l02) -- (u2);
				\draw[thick] (l12) -- (w2);
				\draw[thick] (w2) -- (u2);
				\draw[thick] (l22) -- (u2);
				\draw[thick] (u2) -- (r2);
				\draw[dotted, thick] (u2) -- (v2);
				
				\fill (u2) circle (4pt);
				\fill (w2) circle (4pt);
				
				
				
				\coordinate (T) at (8,0);
				
				\coordinate (v1r) at ($(0,1.2) + (T)$);
				\coordinate (v2r) at ($(1,0.2) + (T)$);
				\coordinate (v3r) at ($(v1r) + (-0.55,0.65)$);
				\coordinate (v4r) at ($(v2r) + (0.9,1)$);
				
				\coordinate (a1r) at ($(-1.0,2.5) + (T)$);
				\coordinate (a2r) at ($(0.4,2.5) + (T)$);
				
				\coordinate (b1r) at ($(1,2.5) + (T)$);
				\coordinate (b2r) at ($(1.8,2.5) + (T)$);
				\coordinate (b3r) at ($(3,2.5) + (T)$);
				
				\coordinate (rr) at ($(1,-1.2) + (T)$);
				
				\draw[thick] (a1r) -- (v1r);
				\draw[thick] (a2r) -- (v1r);
				\draw[thick] (v1r) -- (v2r);
				
				\draw[thick] (b1r) -- (v2r);
				\draw[thick] (b2r) -- (v4r);
				\draw[thick] (b3r) -- (v2r);
				
				\draw[thick] (v2r) -- (rr);
				
				\fill (v1r) circle (5pt);
				\fill (v2r) circle (5pt);
				\fill (v3r) circle (5pt);
				\fill (v4r) circle (5pt);
				
			\end{tikzpicture}
			\caption{Flattening functor example}\label{fig:flat}
		\end{center}
	\end{figure}
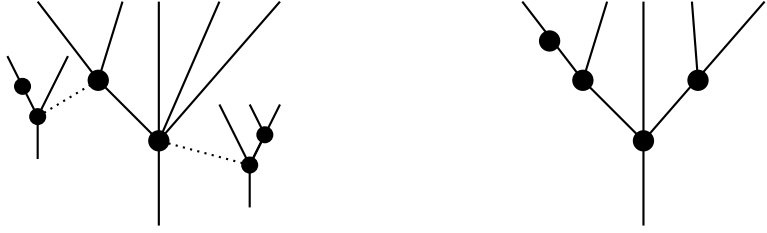
	
	
	\begin{remark}[compatibility with the strict case]\label{rem:strict}
		By Proposition~\ref{pr:dalg} we may also write $\Flat:=\Psi\circ N_{\mathcal{O}_{C}}^{-1}$ for the
		strict analogue of $\Flat^h$, e.g. see figure~\ref{fig:flat}.
		Let $\eps^*\colon\Alg(\mathcal{O}_C,\Top)\to\Alg(W\mathcal{O}_C,\Top)$ be the pullback
		along the augmentation, $\pi^*$ the pullback along
		$\pi\colon\underline{\Omega}\mathcal{O}_C\to\Omega\mathcal{O}_C$, and $\rho^*$ the pullback along
		$\rho\colon\Omega\to\Omega_C$, $K_g\mapsto g$. 
		Then
		\[
		N^h_{\mathcal{O}_C}\circ\eps^*=\pi^*\circ N_{\mathcal{O}_C},\ 
		\Flat^h\circ\pi^*=
		\rho^*\circ\Flat.
		\]
		The first identity is immediate from definition~\ref{def:nhp}: for
		$A=\eps^*A_0$ the map $\varphi_{(T,\mu)}$ is the action of
		$\eps(T)$ and does not depend on $\mu$. 
		For the second, let $X=\pi^*Y$; then $X(j_g,\pi_g(\sigma))=Y(j_g)$ does not depend on $\sigma\in K_g$, and any two factorizations of $g$ are connected by
		refinements, which are composites of inner faces of $\Omega\mathcal{O}_C$
		collapsing each $\mathcal{G}_v(\sigma)$ onto the corolla decorated by $g(C_v)$.
		By Lemma~\ref{lm:ssegal} the presheaf $Y$ takes the same value on all of
		them, namely the one prescribed by $\Psi$. 
	\end{remark}
	
	\section{Applications and concluding remarks}\label{sec:app}
	
	In this section, we provide an algebraic model of cellular chains for any topological operad with CW objects by the above CW-approximation.
	Then we mention several applications of CW-approximation to some well-known topological operads and diagrams, as well as their morphisms.
	
	\begin{corollary}
		Let $O$ be a topological $C$-colored operad with a CW complex in each arity. 
		Then there exists a strict dg-operad $C_{*}^{cw}(O_{cw})$ of cellular chains that is quasi-isomorphic to $C_{*}(O)$.
	\end{corollary}
	\begin{proof}
		By corollary~\ref{corm:rect}, there is the well-defined operad $C_{*}^{cw}(O_{cw})$ in chains and a zigzag of quasi-isomorphisms
		\[
		\begin{tikzcd}
			C_{*}(O) & C_{*}(R)\arrow[swap]{l}{\sim}\arrow{r}{\sim} & C_{*}(O_{cw}).
		\end{tikzcd}
		\]
		By collating it with the standard zigzag of quasi-isomorphisms
		\[
		\begin{tikzcd}
			C_{*}(O_{cw}) & Q\arrow[swap]{l}{\sim}\arrow{r}{\sim} & C_{*}^{cw}(O_{cw}),
		\end{tikzcd}
		\]
		one obtains the claim.
	\end{proof}
	
	There are well-known CW structures on little $n$-disks $\mathcal{E}_{n}$ operad (by its Barratt-Eccles cellular model), in particular, Stasheff associahedra $\mathcal{A}_{\infty}$~\cite{be-fr-04}.
	We proceed with two examples of CW-approximation for operads, which seem to be new.
	
	\begin{example}
		It was proved in~\cite{clrrw-22} that the normalized cacti operad is a homotopy operad, denoted by $\Cact^{1}$.
		A CW structure on the objects of the normalized cacti operad was constructed in~\cite{ka-07}.
		However, the homotopy operad is not CW with respect to Kaufmann's CW structure on its objects.
		By Theorem~\ref{thmm:cwapp}, we obtain a CW homotopy operad and, by rectification (Theorem~\ref{thm:cwapp}), its strict CW  operad $\Cact^{1}_{cw}$, together with a homotopy isomorphism $\Cact^{1}\Rightarrow \Cact^{1}_{cw}$.
	\end{example}
	
	\begin{example}
		The real Fulton-MacPherson operad $FM_{n}$ (i.e. a compactification of the configuration space for pairwise distinct points in $\R^{n}$) consists of manifolds with corners, which admit a natural cell stratification.
		A stratification was constructed  in~\cite{ge-jo-94}, where it was claimed that the operad is cellular.
		However, this was incorrect, as pointed out by Tamarkin.
		In~\cite{sa-22}, it was proved that the real Fulton-MacPherson operad $FM_{2}$ is a CW operad.				
		As in the previous example, we obtain a CW homotopy operad (with respect to a given CW structure on objects, e.g. as in~\cite{ge-jo-94}) $(FM_{n})_{cw}$ and a homotopy isomorphism $FM_{n}\Rightarrow (FM_{n})_{cw}$.
		Our approach bypasses the need for explicit geometric constructions and is model category-free.
	\end{example}
	
	For other interesting examples of topological operads, we mention the operad of framed little disks $fD_{n}$, etc. 
	
	\begin{example}
		A topological $1$-colored homotopy algebra $X$ over $\Ass^{\otimes_{BV} n}$ is called an $E_{n}$-space (e.g. see~\cite{ma-sh-st-00}).
		May's recognition principle~\cite{ma-72} states that a connected topological space $X$ (with homotopy type of a CW complex) admits a structure of a grouplike $E_n$-algebra iff it is weakly equivalent to $n$-th iterated loop space of another space $Y$, i.e. $X\sim \Omega^{n} Y$.
		Any loop space such as $S^{1}\simeq\Omega \C P^{\infty}$, $S^{3}\simeq \Omega \Hc P^{\infty}$ or $\C P^{\infty}\simeq \Omega K(\Z,3)$ has a structure of $E_{1}$-space also called an $A_{\infty}$-space (or classically, a homotopy associative $H$-space).
		The standard two-cell CW structure on a sphere is not preserved by such an $A_{\infty}$-structure.
		However, an application of Theorem~\ref{thmm:cwapp} produces a homotopy isomorphic $A_{\infty}$-structure on the same space that is CW in the above sense.
		Regarding $\C P^{\infty}\simeq K(\Z,2)$, it possesses a strict $H$-space structure (e.g. by considering the realization of a simplicial model for an abelian EML space~\cite{ma-67}).
		One can consider the standard model of $\C P^{\infty}=S^{\infty}/\C^{\times}$, equipped with Segre embeddings for the monoidal structure.
		However, this runs into the problem of building higher homotopies. 
		Namely, a comparison of two compositions leads to permuting homogeneous coordinates up to homotopy.
		Theorem~\ref{thmm:cwapp} resolves this problem by applying it to the already existing algebra and 
		coherently modifying its maps by homotopies, thus resulting in a structure that is compatible with any given CW structure in the above sense.
	\end{example}
	
	\begin{example}
		As noted by Hackney~\cite{ha-18}, open configuration spaces $\Conf_{k} D^{n}$ lack a strict operadic composition and degeneracies, restricting them to a Segal semi-dendroidal space. 
		However, because the fibers of the forgetful map $E_{n}\to \Conf_{k} D^{n}$ are convex open polytopes, $\Conf_{k} D^{n}$ is a strong deformation retract of $E_{n}$. 
		The homotopy invariance of the $W$-construction that is classically established for strong deformation retracts~\cite[Ch.3]{bo-vo-73},~\cite[\S4]{be-mo-06'}, can be executed precisely here via Theorem~\ref{thm:gen} (the generalized form of Theorem~\ref{thmm:cwapp}). 
		Specifically, by viewing the strict operad $E_n$ as a $W\mathcal{O}_{C}$-algebra, Theorem~\ref{thm:gen} transfers this structure along the homotopy equivalence, endowing $\Conf_{k} D^{n}$ with a (CW) $W\mathcal{O}_{C}$-algebra structure and a homotopy isomorphism to $E_n$.
		Applying~\cite[Theorem~B.6]{clrrw-22} (or Theorem~\ref{thmm:preshemb}) to this algebra yields a strict Segal presheaf on $\underline{\Omega} \mathcal{O}_{C}$, and hence on $\Omega \mathcal{O}_{C}$ by applying the flattening functor. 
		Unlike the semi-dendroidal model, this resolves the obstruction to the existence of degeneracies, providing a full albeit (non-canonical) homotopy dendroidal model for open configuration spaces.
	\end{example}
	
	\begin{example}
		By~\cite{be-mo-06'}, diagrams  $I\to \Top$ are in bijection with topological algebras over the unary operad $\Diag I$ in sets, to which the above CW-approximation therefore applies.
		(Notice there is no known satisfactory functorial CW approximation for topological spaces, in general, except for the one given by the Quillen adjunction $(|-|,\Sing)$.) 				
		Given a (strict) diagram $D\colon I\to\Top$ with CW complexes, Theorem~\ref{thmm:cwapp} provides a homotopy isomorphism $D\to D_{cw}$ with a homotopy coherent CW diagram $D_{cw}$.
		After applying rectification, one obtains a zigzag of homotopy equivalences
		\[
		\begin{tikzcd}
			D & 		R\arrow[swap]{l}{\sim}\arrow{r}{\sim} & \widetilde{D}_{cw}
		\end{tikzcd},\ 
		\widetilde{D}_{cw}=\eps_{!}D_{cw}.
		\]
		(Here, notice that the arrows are not strictly invertible, in general. 
		Indeed, the arrows are objectwise homotopy equivalences, whose inverse maps may not commute with the arrows of the diagrams in question.)
		The diagrams $D$, $\widetilde{D}_{cw}$ have homotopy equivalent homotopy colimits, say, $X$.
		We expect that the homotopy invariants of $X$ such as Massey products or formality can be effectively studied in terms of cellular higher homotopies in $\widetilde{D}_{cw}$.
		For one possibly interesting source of examples we indicate diagrams of Eilenberg-MacLane spaces of height $1$ (i.e. of compact tori) whose colimits yield toric varieties. 
	\end{example}	
	
	\begin{example}
		Let $f\colon A\to B$ be a morphism of topological $P$-operads with CW objects.
		In other words, $f$ is a topological algebra over $P\otimes_{BV} \Delta^{1}$ with suitable restrictions to trees with black vertices at leaves ($A$) and root ($B$), respectively, and satisfying strict infinitesimal associativity.
		Then application of Theorem~\ref{thmm:cwapp} yields a homotopy CW morphism $f_{cw}\colon A_{cw}\Rightarrow B_{cw}$ between homotopy CW algebras, related by a homotopy isomorphism $f\Rightarrow f_{cw}$ (i.e. a homotopy algebra over $P\otimes_{BV} \Delta^{1}\otimes_{BV} \Delta^{1}$) restricting to homotopy isomorphisms $A\Rightarrow A_{cw}$ and $B\Rightarrow B_{cw}$.
		Furthermore, if the given algebras $A,B$ are already CW, then one constructs a homotopy CW morphism $f'_{cw}\colon A\Rightarrow B$ related to $f$ by a homotopy isomorphism $f\Rightarrow f'_{cw}$.
		To do so, in the proof of Theorem~\ref{thmm:cwapp} one starts with a given set of trivial higher homotopies.
	\end{example}

	\begin{remark}\label{rem:symcase}
		A similar argument to~\cite[Theorem B.6]{clrrw-22} provides an embedding of homotopy algebras over symmetric colored operads.
		This is a version of Theorem~\ref{thm:halgdiag} (i.e. Theorem~\ref{thmm:preshemb} of the introduction) for symmetric operads.
		The necessary changes to the above proof in the symmetric case consist of fixing the symmetric group labeling of $P$-trees and including the equivariance of the infinitesimal structure with respect to it.
		We leave details to be filled by an interested reader.
		On the other hand, we are not aware of any viable strategy to generalize Theorem~\ref{thm:gen} (i.e. Theorem~\ref{thmm:cwapp}) to symmetric operads using equivariant $G$-CW-approximation with respect to the symmetric group $G$-action.
	\end{remark}

	\begin{remark}
		Notice that Proposition~\ref{pr:dalg} and Theorem~\ref{thm:halgdiag} in the particular case of $P=Q\otimes_{BV}\Delta^{1}$ give an embedding of strict (homotopy, respectively) morphisms of strict (homotopy, respectively) $Q$-algebras to the presheaf category.
		This is a different construction of morphisms from that in Theorem~\ref{thmm:cwapp}, because the image of the embedding lacks an explicit description. 
		Concretely, $\Omega(Q\otimes_{BV}\Delta^{1})$ is different from $\Omega Q\times\Delta^{1}$.
	\end{remark}

\end{document}